\documentclass[12pt]{article}
 \usepackage[latin5]{inputenc}
 \usepackage{amsfonts}
 \usepackage{authblk}
 \usepackage{amsmath}

 \numberwithin{equation}{section}

 \newtheorem{theorem}{Theorem}[section]
 \newtheorem{lemma}[theorem]{Lemma}
 \newtheorem{proposition}[theorem]{Proposition}

\newenvironment{proof}[1][Proof]{\begin{trivlist}
\item[\hskip \labelsep {\bfseries #1}]}{\end{trivlist}}

\newenvironment{remark}[1][Remark]{\begin{trivlist}
\item[\hskip \labelsep {\bfseries #1}]}{\end{trivlist}}

\newcommand{\qed}{\nobreak \ifvmode \relax \else
      \ifdim\lastskip<1.5em \hskip-\lastskip
      \hskip1.5em plus0em minus0.5em \fi \nobreak
      \vrule height0.75em width0.5em depth0.25em\fi}

\usepackage{color}

\begin{document}

 \title{Non-viscous Regularization of the Davey-Stewartson Equations: Analysis and Modulation Theory}
 \date{}
\author[1]{Yanqiu Guo \footnote{guo@math.tamu.edu} }
\author[2]{Irma Hacinliyan\footnote{hacinliy@itu.edu.tr}}
\author[1,3]{Edriss S. Titi\footnote{titi@math.tamu.edu, edriss.titi@weizmann.ac.il}}
\affil[1]{Department of Mathematics, Texas A\&M University, 3368 TAMU, College Station, TX 77843, USA.}
\affil[2]{Department of Mathematics, Istanbul Technical University, 34469 Maslak, Istanbul, Turkey.}
\affil[3]{Department of Computer Science and Applied Mathematics, Weizmann Institute of Science, Rehovot 76100, Israel.}
\renewcommand\Authands{ and }
  \maketitle

 \begin{abstract}
 In the present study we are interested in the Davey-Stewartson equations (DSE) that model packets of surface and capillary-gravity waves. We focus on the elliptic-elliptic case, for which it is known that DSE may develop a finite-time singularity. We propose three systems of non-viscous regularization to the DSE in variety of parameter regimes under which the finite blow-up of solutions to the DSE occurs. We establish the global well-posedness of the regularized systems for all initial data. The regularized systems, which are inspired by the $\alpha$-models of turbulence and therefore are called the $\alpha$-regularized DSE, are also viewed as unbounded, singularly perturbed DSE. Therefore, we also derive reduced systems of ordinary differential equations for the $\alpha$-regularized DSE by using the modulation theory to investigate the mechanism with which the proposed non-viscous regularization prevents the formation of the singularities in the regularized DSE. This is a follow-up of the work \cite{Cao1,Cao2} on the non-viscous $\alpha$-regularization of the nonlinear Schr\"odinger equation.
\end{abstract}


\vspace{0.5cm}

 \section{Introduction}
 The Davey-Stewartson equations (DSE) are given by:
 \begin{align} \label{ds}
 \begin{cases}
  i v_t + \Delta v+\beta |v|^2 v - \rho  \phi_x v=0 \\
  \phi_{xx} +\nu \phi_{yy} =  (|v|^2)_x\\
  v(x,y,0)=v_0(x,y)
 \end{cases}
 \end{align}
 for the spacial variables $(x,y)\in \mathbb R^2$, and the time variable $t\in \mathbb R$, with zero boundary condition at infinity, where the complex-valued function $v(x,y,t)$ represents the amplitude of a wave packet, and the real-valued function $\phi(x,y,t)$ stands for the free long wave mode. This system can be classified as the elliptic-elliptic type for position $\nu$, and the elliptic-hyperbolic type for negative $\nu$. System (\ref{ds}) was first introduced by Davey and Stewartson \cite{dav}, and later by Djordjevic and Redekopp  \cite{red} to model propagation of weakly nonlinear water waves that travels predominantly in one direction, but in which the wave amplitude is modulated slowly in two horizontal directions. System (\ref{ds}) is a Hamiltonian system, which has certain conserved quantities: the $L^2$-energy as well as the Hamiltonian $\cal H$:
 \begin{align}
 {\cal H}(v)=\int_{\mathbb{R}^2}\left[|\nabla v|^2-\frac{\beta}{2} |v|^4+\frac{\rho}{2} \left(\phi_x^2+\nu \phi_y^2 \right)\right]\,dxdy.     \label{hamds}
 \end{align}

Ghidaglia and Saut proved the local well-posedness of the DSE (\ref{ds}) with $\nu>0$ for the initial data $v_0 \in H^1({\mathbb{R}^2})$ in \cite{ghidaglia}.
Moreover, for $\beta \leq \min (\rho,0)$, the solution in the elliptic-elliptic case exists globally in time, whereas for $\beta > \min(\rho,0)$, it has a finite maximum lifespan (cf. \cite{ghidaglia}). Also, the well-posedness and the scattering of a more general and abstract class of the DSE was investigated  in \cite{kuz}.

The ground-state solutions (also known as standing-wave solutions) of the DSE (\ref{ds}) in the elliptic-elliptic case are solutions of the
form $v(x,y,t)=e^{i\lambda t} R(x,y)$ and $\phi(x,y,t)=F(x,y)$, where $R$ and $F$ are real-valued functions with $\lambda>0$.
Accordingly, the ground-state functions $R$ and $F$ satisfy the following coupled nonlinear elliptic eigenvalue problem:
\begin{align}   \label{groundds}
 \begin{cases}
\Delta R - \lambda R +\beta R^3-\rho R F_x=0 \\
 F_{xx} +\nu F_{yy} = (R^2)_x
 \end{cases}
 \end{align}
where $\nu>0$, with zero boundary condition at infinity. The existence of ground-state solutions was established by Cipolatti in \cite{cip}. An alternative way of characterizing the
solution of (\ref{groundds}) is presented in \cite{pap} and it is shown that the solution of the DSE (\ref{ds}) exists globally in time provided
that the initial value $v_0\in H^1(\mathbb R^2)$ satisfies $\|v_0\|_{L^2({\mathbb{R}^2})}<\|R\|_{L^2({\mathbb{R}^2})}$ where $R$ is the ground-state solution of (\ref{groundds}). In \cite{ablowitz}, Ablowitz et. al. explored necessary conditions for wave collapse in the DSE (\ref{ds}) by using the global existence theory and numerical calculations of the ground-state.

The aim of our paper is to introduce three special non-viscous, Hamiltonian regularizations to the nonlinear terms in the elliptic-elliptic DSE (\ref{ds}) in various parameter regimes, and establish the global well-posedness of these regularized systems. These regularizations are in the spirit of the $\alpha-$models of turbulence. We will follow the approach in \cite{Cao1,Cao2} in which an $\alpha$-regularized nonlinear Schr\"odinger equation (NLS) was investigated. See also references to the $\alpha$-models of turbulence in \cite{Cao1}.

 The two-dimensional cubic NLS equation is given by:
 \begin{align}    \label{nls}
 i v_t+\Delta v+|v|^2v=0
 \end{align}
 with the initial condition $v(x,y,0)=v_0(x,y)$, where $v$ is a complex-valued function. It is a model for the propagation of a laser beam in an optical Kerr medium, or a model for water waves at the free surface of an ideal fluid as well as plasma waves (see, e.g., \cite{kel,sul} and references therein).
 It is well known that the 2d cubic NLS (\ref{nls}) blows up in finite time (see, e.g., \cite{car, cav1,cav,gin,gls,kat,sul,wei} and references therein).
 Notice that the 2d cubic NLS is the deep water limit of the DSE. On the other hand, the DSE can be regarded as a perturbation of the 2d cubic NLS, and this perturbation does not effect the blow up rate \cite{fib,pap}.

In \cite{Cao1,Cao2}, the following non-viscous regularized system of the cubic NLS equation (\ref{nls}) is investigated:
\begin{align}   \label{nlsh}
 \begin{cases}
 i v_t+\Delta v+u v=0   \\
 u-\alpha^2 \Delta u=|v|^2
\end{cases}
\end{align}
with the initial condition $v(x,y,0)=v_0(x,y)$, with zero boundary condition at infinity, where $\alpha>0$ is the regularization parameter. Notice that when $\alpha=0$, (\ref{nlsh}) reduces to (\ref{nls}).
It is shown in \cite{Cao1} that the Cauchy problem (\ref{nlsh}) is globally well-posed.
Moreover, by regarding system (\ref{nlsh}) as a perturbation of the cubic NLS equation (\ref{nls}), and by adopting the modulation theory, different scenarios are demonstrated in \cite{Cao2} of how the regularization prevents the formation of the singularities of the cubic NLS equation.

This paper consists of five sections. Section 2 introduces notations, and summarizes some embedding and interpolation theorems, as well as properties of certain elementary operators. In section 3, we briefly introduce three different non-viscous Helmholtz type of $\alpha$-regularizations to the DSE in the elliptic-elliptic case and state the global well-posedness of these $\alpha$-regularized systems. In section 4, we prove the local well-posedness of these $\alpha$-regularized systems by a fixed point argument, as well as the extension to global solutions by using the conservation of the $L^2$-energy and the Hamiltonian. In section 5, we apply modulation theory following ideas from \cite{Cao2, fib, pap, sul}, to shed light on the mechanism of how these regularizations prevent the formation of the singularities in the regularized DSE.

\section{Notations and preliminaries}
The following notations are used throughout the paper.
 \begin{eqnarray*}
 &&\hspace{-0.5cm} \Delta=\partial_{xx}+\partial_{yy}, ~~\Delta_\nu=\partial_{xx}+\nu \partial_{yy} ; \\
 &&\hspace{-0.5cm}  L^q=L^q(\mathbb{R}^2),\;\;\| \cdot \|_q \;\;\text{denotes}\;\; L^q-\textrm{norm};   \\
 && \hspace{-0.5cm} H^q=H^q(\mathbb{R}^2),\;\; \| \cdot \|_{H^q}\;\;\text{denotes}\;\; H^q-\textrm{Sobolev norm};\\
 &&\hspace{-0.5cm} W^{2,p}=W^{2,p}(\mathbb{R}^2),\;\; \|\cdot\|_{W^{2,p}}\;\;\text{denotes}\;\;W^{2,p}-\textrm{Sobolev norm}; \\
 &&\hspace{-0.5cm} L^q_tL^r_{\bf z}=L^q(I;L^r)~ (I=[0,T],\,{\bf z}=(x,y)),\;\; \|\cdot\|_{r,q} \;\; \text{denotes}\;\; L^q_tL^r_{\bf z}-\textrm{norm}.
 \end{eqnarray*}

Next, we recall some classical two-dimensional Gagliardo-Nirenberg and Sobolev inequalities, as well as elementary interpolation estimates (see, e.g., \cite{adm}):
 \begin{eqnarray}
 &\textrm{(1)} & \|v\|_{q}\leq C\|v\| _2^{\frac{2}{q}}\, \|v\|_{H^1}^{\frac{q-2}{q}} \hspace{1cm} \textrm{for} \,v\in H^1,\,\, 0<{\frac{q-2}{q}}\leq 1 \label{inq1}\\
 &\textrm{(2)} & \|v\|_{q}\leq C\|v\| _{W^{2,p}} \hspace{1cm} \textrm{for}\, v\in W^{2,p},\,\, 1<p\leq q, \label{inq2}\\
 &\textrm{(3)} & \|v\|_{q}\leq C\|v\| _{H^1} \hspace{1cm} \textrm{for}\, v\in H^1,\,\, 2\leq q<\infty, \label{inq3}\\
 &\textrm{(4)} & \|v\|_{q}\leq C\|v\| _{H^2} \hspace{1cm} \textrm{for}\, v\in H^2,\,\, 2\leq q\leq\infty, \label{inq4}\\
 &\textrm{(5)} & \|v\|_{2q}\leq C\|v\|_{H^k} \hspace{1cm} \textrm{for}\, v\in H^{k},\,\, k=(q-1)/ q<2, \label{inq5}\\
 &\textrm{(6)} & \|v\|_{H^k}\leq \|v\|_{H^2}^{\frac{k}{2}} \|v\| _{2}^{\frac{2-k}{2}} \hspace{1cm} \textrm{for}\, v\in H^{2},\,\, k<2. \label{inq6}
 \end{eqnarray}

In addition, for the elliptic Helmholtz equation $\psi-\alpha^2 \Delta \psi= \Psi$, its solution will be denoted as $\psi=B(\Psi)$ where
 \begin{align}  \label{def-B}
 B=(\text{Id}-\alpha^2 \Delta)^{-1},
 \end{align}
 where Id represents the identity operator. By Plancherel identity, for $\Psi\in L^2$, one has
 \begin{align}  \label{L2b}
 \|B(\Psi)\|_2  \leq \|\Psi\|_2.
 \end{align}
 Also, for $\Psi\in L^p$, $1<p<\infty$, the following regularity property of elliptic operators is standard (see, e.g. \cite{mit,ste,yud1,yud2}):
 \begin{equation}
   \|B(\Psi)\|_{W^{2,p}}\leq C_{\alpha,p} \|\Psi\|_p, \;\;\;\; \textrm{for} \,\,1<p<\infty, \label{besinq}
 \end{equation}
 where $C_{\alpha,p}$ depends on $\alpha$ and $p$, and $C_{\alpha,p}\sim1/\alpha^2$, as $\alpha\rightarrow0^{+}$.

 Moreover, the Poisson-like equation $\Delta_\nu\psi=\Psi_x$, for $\nu>0$, can be solved in terms of $\Psi$, and we denote by, $\psi_x=E(\Psi)$, where the singular integral operator $E$ is defined via the Fourier transform by
 \begin{equation}
 \widehat{E(f)}(\xi_1,\xi_2)=\frac{\xi_1^2}{\xi_1^2+\nu \xi_2^2}\widehat{f}(\xi_1,\xi_2). \label{eop}
 \end{equation}
 Once again, due to Plancherel identity, for $\Psi\in L^2$, one has
 \begin{align}  \label{L2bb}
 \|E(\Psi)\|_2 \leq \|\Psi\|_2.
 \end{align}
 Also, since the operator $E$ is of order zero, then by the Calderon-Zygmund theorem (see, e.g, \cite{ste,fol}) we have
 \begin{equation}
  \|E(\Psi)\|_p \leq C_p \|\Psi\|_p, \,\,\,\; \textrm{for} \,\,1<p<\infty, \label{posinq}
 \end{equation}
 where $C_p$ depends on $p$.

As usual, throughout the paper, the constant $C$ may vary from line to line.

 \section{Helmholtz $\alpha$-regularized Davey-Stewartson \\ equations}  \label{sec-model}

 In this section, inspired by the inviscid $\alpha$-regularization of the cubic NLS introduced in \cite{Cao1,Cao2} (see also references therein), we propose three different regularizations of the DSE (\ref{ds}) of the elliptic-elliptic type (i.e. $\nu>0$) in the parameter regime $\beta>\min(\rho,0)$ where the finite-time blow-up takes place \cite{ghidaglia}. We also state the global well-posedness of these $\alpha$-regularized systems.

 \subsection{Case 1: $\rho>0$ and $\beta>0$}
 Under this scenario, by the conservation of the Hamiltonian (\ref{hamds}) of DSE (\ref{ds}), we see that the cubic nonlinearity $\beta |v|^2 v$ in (\ref{ds}) tends to amplify the $H^1$-norm, while the nonlocal term $-\rho \phi_x v$ can be viewed as a dissipation. Consequently, the finite-time blow-up of the $H^1$-norm of the DSE (\ref{ds}) is caused by the growth of the local term $\beta |v|^2 v$, which should be regularized to guarantee global existence in $H^1$. As a result, we introduce the first $\alpha$-regularized Davey-Stewartson equations (RDS1):
 \begin{align}   \label{dsh1}
 \begin{cases}
  i v_t + \Delta v+\beta u v-\rho \phi_{x}v=0, \hspace{0.5cm} \Delta_\nu \phi=(|v|^2)_x,\\
  \hspace{1.5cm} u-\alpha^2 \Delta u=|v|^2, \\
  \hspace{1.5cm} v(x,y,0)=v_0(x,y),
 \end{cases}
 \end{align}
 where $\nu>0$, $\alpha>0$, $\rho>0$ and $\beta>0$. Notice that system (\ref{dsh1}) reduces to the DSE (\ref{ds}) when $\alpha=0$.
 Formally, system (\ref{dsh1}) has two conserved quantities:  the $L^2$-energy and the Hamiltonian,
 \begin{align}   \label{hamds1}
 {\cal H}_1(v) = \int_{\mathbb{R}^2}\left[|\nabla v|^2-\frac{\beta}{2} u|v|^2 +  \frac{\rho}{2} \left(\phi_x^2+\nu \phi_y^2 \right)\right]\,dxdy.
 \end{align}

The RDS1 system (\ref{dsh1}) is globally well-posed in $H^1$. In particular, we have
\begin{theorem}  \label{global1}
Let $v_0\in H^1$, then there exists a unique global solution of system RDS1 (\ref{dsh1}), for all $t\in \mathbb R$, such that $v \in {\cal C}(\mathbb R,H^1)\cap \, {\cal C}^1(\mathbb R,H^{-1})$, and $\nabla \varphi \in {\cal C}(\mathbb R,L^p)$, for $p>1$. Moreover, the energy ${\cal N}(v)=\|v\|_2^2$  and the Hamiltonian $\mathcal H_1(v)$ are conserved in time. In addition, the solution depends continuously on the initial data.
\end{theorem}

 \subsection{Case 2: $\rho<\beta<0$}
 In this case, by the structure of the Hamiltonian (\ref{hamds}) of DSE (\ref{ds}), we notice that the nonlocal term $-\rho \phi_x v$ in DSE (\ref{ds}) may amplify the $H^1$-norm, while the nonlinearity $\beta |v|^2 v$ can be considered as a dissipation. Furthermore, since $\rho<\beta<0$, the nonlocal term overcomes the cubic nonlinearity, leading to a finite-time blow-up \cite{ghidaglia}. Therefore, in order to obtain the global existence of solutions, the nonlocal term $-\rho v \phi_x $ should be smoothed. We introduce the second $\alpha$-regularized Davey-Stewartson equations (RDS2) as follows:
 \begin{align} \label{dsh2}
 \begin{cases}
  i v_t + \Delta v+\beta |v|^2 v-\rho \varphi_{x}v=0,\hspace{0.5cm} \Delta_\nu \psi=u_x,  \\
  \hspace{0.5cm} u-\alpha^2 \Delta u=|v|^2, \hspace{0.5cm}\varphi-\alpha^2 \Delta \varphi=\psi,   \\
  \hspace{1.5cm} v(x,y,0)=v_0(x,y),
 \end{cases}
 \end{align}
 where $\nu>0$, $\alpha>0$ and $0>\beta>\rho$.
 Here the RDS2 system (\ref{dsh2}) reduces to the DSE (\ref{ds}) when $\alpha=0$. Formally, system (\ref{dsh2}) has two conserved quantities: the $L^2$-energy and the Hamiltonian:
 \begin{eqnarray}
 {\cal H}_2(v) = \int_{\mathbb{R}^2}\left[|\nabla v|^2 - \frac{\beta}{2} |v|^4 + \frac{\rho}{2} \left(\psi_x^2+\nu \psi_y^2 \right)\right]\,dxdy. \label{hamds2}
 \end{eqnarray}

The following result states that the RDS2 (\ref{dsh2}) is globally well-posed in $H^1$.
\begin{theorem}  \label{global2}
Let $v_0\in H^1$, then there exists a unique global solution of the system RDS2 (\ref{dsh2}), for all $t\in \mathbb R$, such that $v \in {\cal C}(\mathbb R,H^1)\cap \, {\cal C}^1(\mathbb R,H^{-1})$, and $\nabla \varphi \in {\cal C}(\mathbb R,W^{4,p})$ for $p>1$. Moreover, the energy ${\cal N}(v)=\|v\|_2^2$  and the Hamiltonian $\mathcal H_2(v)$ are conserved in time. In addition, the solution depends continuously on the initial data.
\end{theorem}

 \subsection{Case 3: $\rho<0$ and $\beta>0$}
 Notice that each of the two nonlinear terms individually may cause the blow-up of DSE (\ref{ds}), and thus both of them should be smoothed in order to prevent the development of singularity. As a result, the third $\alpha$-regularized Davey-Stewartson equations (RDS3) is given by:
 \begin{align} \label{dsh3}
 \begin{cases}
  i v_t + \Delta v+\beta u v-\rho \varphi_{x}v=0,\hspace{0.5cm} \Delta_\nu \psi=u_x, \\
  \hspace{0.5cm} u-\alpha^2 \Delta u=|v|^2, \hspace{0.5cm} \varphi-\alpha^2 \Delta \varphi=\psi,\\
  \hspace{1.5cm} v(x,y,0)=v_0(x,y),
 \end{cases}
 \end{align}
 where $\nu>0$, $\alpha>0$, $\rho<0$ and $\beta>0$.
 As in previous cases, the RDS3 system (\ref{dsh3}) reduces to the DSE (\ref{ds}) when $\alpha=0$. Formally, system (\ref{dsh3}) has two conserved quantities: the $L^2-$energy and the Hamiltonian:
 \begin{eqnarray}
 && {\cal H}_3(v) = \int_{\mathbb{R}^2}\left[|\nabla v|^2 -  \frac{\beta}{2} u|v|^2 + \frac{\rho}{2} \left(\psi_x^2+\nu \psi_y^2   \right)\right]\,dxdy. \label{hamds3}
 \end{eqnarray}

The following theorem states that the RDS3 (\ref{dsh3}) is globally well-posed in $H^1$.
\begin{theorem}  \label{global3}
Let $v_0\in H^1$, then there exists a unique global solution of system RDS3 (\ref{dsh3}), for all $t\in \mathbb R$, such that $v \in {\cal C}(\mathbb R,H^1)\cap \, {\cal C}^1(\mathbb R,H^{-1})$, and $\nabla \varphi \in {\cal C}(\mathbb R,W^{4,p})$, for $p>1$. Moreover, the energy ${\cal N}(v)=\|v\|_2^2$  and the Hamiltonian $\mathcal H_3(v)$ are conserved in time. In addition, the solution depends continuously on the initial data.
\end{theorem}

\section{Proof of the global well-posedness of the $\alpha$-regularized Davey-Stewartson equations}
This section is devoted to prove the global well-posedness of the various $\alpha$-regularized Davey-Stewartson equations proposed in section \ref{sec-model}.
The proof for all the three regularized systems can be presented in a similar manner, so we only demonstrate the proof for Theorem \ref{global3}, i.e. for the system RDS3 (\ref{dsh3}) in the case: $\rho<0$ and $\beta > 0$.
As we have discussed in section \ref{sec-model}, under this scenario, both nonlinear terms $\beta |v|^2 v$ and $-\rho v \phi_x$ in the DSE (\ref{ds}) are regularized. In the following subsections, we shall study the local existence and uniqueness of solutions to (\ref{dsh3}) in $H^1$ and $H^2$, the continuous dependence on initial data in $H^1$, energy and Hamiltonian conservation, as well as the extension to global solutions in $H^1$. In order to make sure that the proof can be readily adjusted to handle the systems RDS1 (\ref{dsh1}) and RDS2 (\ref{dsh2}) as well, we intentionally avoid using the smoothing property of the $\alpha$-regularization operator in justifying the local well-posedness. The smoothing property is solely used when studying the extension to global solutions.

\subsection{Short-time existence and uniqueness of solutions in $H^1$}   \label{short}
We follow the approach in \cite{ghidaglia,kat} to establish short-time existence and uniqueness of solutions to the RDS3 system (\ref{dsh3}) by using a fixed-point argument. In particular, we will prove the following theorem:
\begin{theorem}  \label{localf1}
Let $v_0\in H^1$, then there exists a unique solution of the RDS3 system (\ref{dsh3}) on $I=[0,T]$, for some $T(\|v_0\|_{H^1})>0$, such that $v \in {\cal C}(I,H^1) \cap \, {\cal C}^1(I,H^{-1})$, and $\nabla \varphi \in {\cal C}(I,W^{4,p})$, for $p>1$. Moreover, the energy ${\cal N}(v)=\|v\|_2^2$ is conserved on $[0,T]$.
\end{theorem}

To begin with, by using the operators $B$ and $E$ defined in (\ref{def-B}) and (\ref{eop}), respectively, we write the RDS3 system (\ref{dsh3}) as
\begin{align}  \label{dsh1'}
i v_t + \Delta v + F(v)=0,
\end{align}
where the nonlinearity
\begin{equation}
F(v)=\beta B(|v|^2)v-\rho B(E(B(|v|^2))) v, \label{f1}
\end{equation}
where $\beta>0$ and $\rho<0$. Next, by Duhamel's principle, we convert equation (\ref{dsh1'}) into an equivalent integral equation
 \begin{equation}
  v(t)=G_0v_0+i G \circ   F(v) \label{int1}
 \end{equation}
 where $G_0,$ $G$ are linear operators given by
\begin{equation}
(G_0w)(t) = e^{it \Delta}w,\hspace{0.5cm}  (Gf)(t)=\int_{0}^t  e^{i(t-s)\Delta} f(s)\, ds.\label{G0G1}
\end{equation}
Some well-known properties of the operators $G_0$ and $G$ are given in the Appendix A.

Before proving Theorem \ref{localf1}, we will study the properties of the maps $F$ and $G \circ F$.
Set $\mathbf {z}=(x,y)\in \mathbb R^2$, and let $t\in [0,T]$. We introduce the following function spaces:
 \begin{align}   \label{space-X}
 X=L^{\infty}_t L^2_{\mathbf z} \cap  L^4_t L^4_{\mathbf z} \text{\;\;and\;\;}  X_0=L^{\infty}_t L^2_{\mathbf z} \cap  L^{\infty}_t L^4_{\mathbf z} \subset X,
 \end{align}
 with their relevant norms
 \begin{align*}
 \|v\|_{X}=\max\{\|v\|_{2,\infty},\|v\|_{4,4} \}  \text{\;\;and\;\;}    \|v\|_{X_0}=\max \{\|v\|_{2,\infty},\|v\|_{4,\infty}\}.
 \end{align*}
Also, we denote by $B_R(X_0)$ the closed ball in $X_0$, with center at 0 and radius $R$, i.e.,
\begin{align*}
B_R(X_0) =  \{v\in X_0: \|v\|_{X_0}   \leq R  \}.
\end{align*}

The following result states some properties of the nonlinear operator $G \circ F$.
\begin{proposition} \label{Gf1}
Let $T>0$ be given. The nonlinear operator $G \circ  F: X_0 \rightarrow X$ is bounded and locally Lipschitz continuous. Moreover, on each ball $B_R(X_0)$, $G \circ F$  is a contraction mapping in the metric of $X$, provided $T$ is sufficiently small.
\end{proposition}

 \begin{proof}
 Recall from (\ref{f1}) that $F(v)=\beta B(|v|^2)v- \rho B(E(B(|v|^2)))v$, where $\beta>0$ and $\rho<0$, and the operators $B$ and $E$ defined in (\ref{def-B}) and (\ref{eop}), respectively. By using H\"older's inequality and the properties of $B$ and $E$ given in (\ref{L2b}) and (\ref{L2bb}), respectively, we have
 \begin{eqnarray} \label{f14.3}
  \|F(v)\|_{\frac{4}{3}}&\leq& \beta \|B(|v|^2)\|_{2}\|v\|_{4}+|\rho|\|B(E(B(|v|^2)))\|_2\|v\|_{4}    \nonumber\\
                  &\leq& \beta  \||v|^2\|_{2}\|v\|_{4}+|\rho| \||v|^2\|_2\|v\|_{4}  \leq (\beta+|\rho|) \|v\|_4^3.
 \end{eqnarray}
By Lemma \ref{lgog1} in Appendix A, as well as inequality (\ref{f14.3}), we have
 \begin{align} \label{f14.3a}
 \|G \circ F(v)\|_X &=\max\{   \|G \circ F(v)\|_{2,\infty}, \|G \circ F(v)\|_{4,4}  \} \notag\\
 &\leq \gamma \|F(v)\|_{\frac{4}{3},\frac{4}{3}} \leq \gamma T^{\frac{3}{4}} (\beta+|\rho|) \|v\|_{4,\infty}^3 \leq \gamma T^{\frac{3}{4}} (\beta+|\rho|) \|v\|_{X_0}^3.
 \end{align}
 Consequently, the nonlinear operator $G \circ F: X_0 \rightarrow X$ is bounded.

Next, we show that $G \circ F$ is a continuous operator mapping from $X_0$ into $X$, and on each ball $B_R(X_0)$ the operator $G \circ F$ is a contraction mapping, with respect to the norm of $X$, provided $T$ is sufficiently small. To this end, we let  $v$, $w\in B_R(X_0)$, i.e. $\max\{\|v\|_{2,\infty},\|v\|_{4,\infty},\|w\|_{2,\infty},\|w\|_{4,\infty}\} \leq R$.
 Since $G$ is linear, we use Lemma \ref{lgog1} to obtain
 \begin{align}    \label{lip}
 \|G \circ F(v)-G \circ F(w)\|_{X}=\|G[F(v)-F(w)]\|_{X} \leq   \gamma   \|F(v)-F(w)\|_{\frac{4}{3},\frac{4}{3}}.
 \end{align}
 We decompose $\|F(v)-F(w)\|_{\frac{4}{3},\frac{4}{3}}$ as
 \begin{align}    \label{f-1}
 \|F(v)-F(w)\|_{\frac{4}{3},\frac{4}{3}} \leq \beta(I_1+I_2)+|\rho|(I_3+I_4),
 \end{align}
and claim
 \begin{align}
 & I_1:=\|B(|v|^2-|w|^2)v\|_{\frac{4}{3},\frac{4}{3}}\leq 4 R^2  \min\{T^{\frac{1}{2}} \|v-w\|_{4,4},  \;T^{\frac{3}{4}} \|v-w\|_{4,\infty}\}, \label{I1.1}\\
 & I_2:=\|B(|w|^2)(v-w)\|_{\frac{4}{3},\frac{4}{3}}\leq R^2 \min\{T^{\frac{1}{2}} \|v-w\|_{4,4},  \;T^{\frac{3}{4}} \|v-w\|_{4,\infty}\} ,\label{I1.2}\\
 & I_3:=\|B(E(B(|v|^2-|w|^2)))v\|_{\frac{4}{3},\frac{4}{3}}\leq 4 R^2 \min\{T^{\frac{1}{2}} \|v-w\|_{4,4},  \;T^{\frac{3}{4}} \|v-w\|_{4,\infty}\} ,\label{I1.3}\\
 & I_4:=\|B(E(B(|w|^2)))(v-w)\|_{\frac{4}{3},\frac{4}{3}}\leq R^2 \min\{T^{\frac{1}{2}} \|v-w\|_{4,4},  \;T^{\frac{3}{4}} \|v-w\|_{4,\infty}\}.   \label{I1.4}
 \end{align}

 All of the inequalities (\ref{I1.1})-(\ref{I1.4}) can be proved in a similar manner, so we just demonstrate the proof of (\ref{I1.3}).
 By H\" older inequality, as well as (\ref{L2b}) and (\ref{L2bb}), we have
 \begin{align*}
 I_3^{\frac{4}{3}} &\leq \int_0^T\|B(E(B(|v|^2-|w|^2)))\|_{2}^{\frac{4}{3}}\|v\|_4^{\frac{4}{3}}\,dt \notag\\
 &\leq \int_0^T\||v|^2-|w|^2\|_{2}^{\frac{4}{3}}\|v\|_4^{\frac{4}{3}} \, dt  \notag\\
 &\leq \int_0^T \||v|+|w|\|_{4}^{\frac{4}{3}} \|v-w\|_{4}^{\frac{4}{3}} \|v\|_4^{\frac{4}{3}}  \, dt \notag\\
 &\leq (\|v\|_{4,\infty}+\|w\|_{4,\infty})^{\frac{8}{3}} \min\{  T^{\frac{2}{3}} \|v-w\|_{4,4}^{\frac{4}{3}}, T \|v-w\|_{4,\infty}^{\frac{4}{3}}  \},
 \end{align*}
 which implies (\ref{I1.3}) due to the fact $\|v\|_{4,\infty}+\|w\|_{4,\infty} \leq 2R$.

Combining (\ref{f-1}) and (\ref{I1.1})-(\ref{I1.4}) gives us
\begin{align}   \label{f1b}
\|F(v)-F(w)\|_{\frac{4}{3},\frac{4}{3}} \leq 5 (\beta+|\rho|) R^2   \min\{T^{\frac{1}{2}} \|v-w\|_{4,4},  \;T^{\frac{3}{4}} \|v-w\|_{4,\infty}\}.
\end{align}
By (\ref{lip}) and (\ref{f1b}) it follows that
\begin{align}
&\|G \circ F(v)-G \circ F(w)\|_{X} \leq 5  \gamma (\beta+|\rho|)  R^2 T^{\frac{3}{4}} \|v-w\|_{X_0}, \label{li-1}\\
&\|G \circ F(v)-G \circ F(w)\|_{X} \leq 5 \gamma (\beta+|\rho|)  R^2 T^{\frac{1}{2}} \|v-w\|_X.     \label{li-2}
\end{align}
Notice that (\ref{li-1}) implies that $G \circ F:X_0 \rightarrow X$ is locally Lipschitz continuous. Also, (\ref{li-2}) shows that on each ball $B_R(X_0)$, the operator $G \circ F$ is a contraction mapping with respect to the metric of $X$, provided $T<1/(5 \gamma (\beta+|\rho|) R^2)^2$.
 \hfill $\Box$
 \end{proof}

\bigskip

Next, we introduce the following spaces:
\begin{align*}
Y=\{ v\in X: \; \nabla v  \in X \}    \subset  L^{\infty}(I,H^1), \text{\;\;where\;\;} X=L^{\infty}_t L^2_{\mathbf z}\cap L^4_t L^4_{\mathbf z},
\end{align*}
where $I=[0,T]$, with the norms
\begin{align*}
\|v\|_X= \max\{\|v\|_{2,\infty},\; \|v\|_{4,4}\} \text{\;\;and\;\;} \|v\|_{Y}=\max\{\|v\|_X,\;  \|\nabla v\|_X \}.
\end{align*}
Also, we set
\begin{align}   \label{Yprime}
Y'=\{ f\in X': \; \nabla f  \in X' \} , \text{\;\;where\;\;} X'=L^1_t L^2_{\mathbf z} + L^{\frac{4}{3}}_t L^{\frac{4}{3}}_{\mathbf z},
\end{align}
with the norms
\begin{align*}
\|f\|_{X'}=\inf\{\|g\|_{2,1}+\|h\|_{\frac{4}{3},\frac{4}{3}}:  \, f=g+h \}   \text{\;\;and\;\;}    \|f\|_{Y'}=\max\{ \|f\|_{X'},\;  \|\nabla f\|_{X'} \}.
\end{align*}

Recall the nonlinear operator $F$ is defined in (\ref{f1}) by $F(v)=\beta B(|v|^2)v -  \rho  B(E(B(|v|^2))) v $, where $\beta>0$ and $\rho<0$. Then the following result holds for $F$.
\begin{proposition} \label{pf1}
 The nonlinear operator $F: Y \rightarrow Y'$ is bounded satisfying
\begin{equation}
   \|F(v)\|_{Y'}\leq C(\beta+|\rho|) T^{\frac{3}{4}} \|v\|^3_Y, \hspace{1cm} \textrm{for\;\;} \,v\in Y. \label{f1Y1}
 \end{equation}
 \end{proposition}

 \begin{proof}
 Let $v\in Y$, i.e., $v\in X$ with $\nabla v\in X$. We aim to show that $F(v)\in X'$ and $\nabla F(v)\in X'$ such that $\max\{\|F(v)\|_{X'},\; \|\nabla F(v)\|_{X'}    \}\leq C(\beta+|\rho|) T^{\frac{3}{4}} \|v\|_Y^3  $.
 By virtue of (\ref{f14.3}), one has $F(v)\in L_t^{\frac{4}{3}}L_{\bf z}^{\frac{4}{3}}$ such that
 \begin{align}   \label{f10}
  \|F(v)\|_{\frac{4}{3},\frac{4}{3}} \leq T^{\frac{3}{4}} (\beta +|\rho| ) \|v\|_{4,\infty}^3
  \leq T^{\frac{3}{4}} (\beta+|\rho|) \|v\|_{X_0}^3.
 \end{align}
 Notice that $Y\subset L^{\infty}(I, H^1) \subset X_0=L^{\infty}_t L^2_{\mathbf z} \cap L^{\infty}_t L^4_{\mathbf z}$ due to the imbedding $H^1 \hookrightarrow L^4$.
 Thus $\|v\|_{X_0}\leq C\|v\|_Y$, and along with (\ref{f10}), we deduce
 \begin{align} \label{f1x}
 \|F(v)\|_{X'} \leq \|F(v)\|_{\frac{4}{3},\frac{4}{3}} \leq T^{\frac{3}{4}} (\beta + |\rho| )  \|v\|_{X_0}^3  \leq  C  T^{\frac{3}{4}} (\beta + |\rho| )  \|v\|_{Y}^3.
 \end{align}

 Next, we show that $\nabla F(v)\in X'$. We denote $\tau_h$ the spatial translation by $h\in \mathbb{R}^2$, i.e., $\tau_h v(x)=v(x+h)$. Note that the function spaces considered are translation invariant in spatial variables. Denote the identity operator by Id, then applying (\ref{f1b}) gives us
 \begin{eqnarray}  \label{tau1}
 \|(\tau_h-\text{Id})F(v)\|_{\frac{4}{3}, \frac{4}{3}}&=&\|F(\tau_h v)-F(v)\|_{\frac{4}{3}, \frac{4}{3}}\nonumber\\
 &\leq& 5 (\beta + |\rho|) T^{\frac{1}{2}} \|v\|_{X_0}^2 \|(\tau_h-\text{Id})v\|_{4,4}.
 \end{eqnarray}
Now, dividing (\ref{tau1}) by $|h|$, and then taking the limit as $|h|\rightarrow0$ gives
 \begin{eqnarray}  \label{tau2}
\|\nabla F(v)\|_{X'} \leq \|\nabla F(v)\|_{\frac{4}{3}, \frac{4}{3}} \leq 5 (\beta + |\rho|) T^{\frac{1}{2}} \|v\|_{X_0}^2 \|\nabla  v\|_{4,4}
\leq C(\beta + |\rho| ) T^{\frac{1}{2}} \|v\|_{Y}^3.
 \end{eqnarray}
 Estimate (\ref{f1Y1}) follows from equations (\ref{f1x}) and (\ref{tau2}).
\hfill $\Box$
\end{proof}

In order to prove Theorem \ref{localf1}, for each $v_0\in H^1$, we define operator $\mathcal T: Y \rightarrow Y$ by
\begin{equation*}
\mathcal T(v)=G_0v_0+i G \circ F(v).
\end{equation*}
Since $v_0\in H^1$, we have $G_0v_0\in Y$ due to Lemma \ref{lgbound}. Then, we define
\begin{align}
B_R(G_0 v_0,Y)=\{v\in Y : \|v- G_0 v_0\|_Y  \leq R \}.
\end{align}
The following result states a contraction mapping property of $\mathcal T$.
\begin{lemma} \label{lphe1}
Let $v_0\in H^1$ and $R>0$ be fixed. Then there exists $T (\|v_0\|_{H^1}, R)>0$ sufficiently small so that $\mathcal T: B_R(G_0 v_0,Y)\rightarrow B_R(G_0 v_0,Y)$ is a contraction mapping in the metric of the space $X$.
\end{lemma}
\begin{proof}
Let $v\in B_R(G_0 v_0,Y)$, then by Lemma \ref{lgog1}, Proposition \ref{pf1} and Lemma \ref{lgbound}, we deduce
\begin{align*}
&\|\mathcal T(v)-G_0 v_0\|_Y=\|G \circ F(v)\|_Y \leq \gamma \|F(v)\|_{Y'}\leq \gamma C(\beta+|\rho|) T^{\frac{3}{4}} \|v\|^3_Y \notag\\
&\;\;\;\; \leq C(\beta+|\rho|) T^{\frac{3}{4}}(\|v-G_0 v_0\|_Y+\|G_0 v_0\|_Y)^3 \\
&\;\;\;\; \leq  C(\beta+|\rho|) T^{\frac{3}{4}}(R+\|G_0 v_0\|_Y)^3 \leq C(\beta+|\rho|) T^{\frac{3}{4}}(R+c\|v_0\|_{H^1})^3 <R,
\end{align*}
provided $T$ is sufficiently small, i.e., $T < [R C^{-1} (\beta+|\rho|)^{-1} (R+c\|v_0\|_{H^1})^{-3}]^{\frac{4}{3}} $. This shows that $\mathcal T$ maps $B_R(G_0 v_0,Y)$ into $B_R(G_0 v_0,Y)$, if $T$ is sufficiently small.

Next, we show that $\mathcal T: B_R(G_0 v_0,Y) \rightarrow B_R(G_0 v_0,Y)$ is a contraction mapping. Let $v\in B_R(G_0 v_0,Y)$, i.e., $\|v-G_0 v_0\|_Y\leq R$. It follows that
\begin{align}    \label{R1}
\|v\|_{X_0} &\leq \|v-G_0 v_0\|_{X_0}+\|G_0 v_0\|_{X_0} \notag \\
&\leq C(\|v-G_0 v_0\|_Y+ \|G_0 v_0\|_Y) \leq C(R + c\|v_0\|_{H^1}) =: R_1,
\end{align}
which shows that $v\in B_{R_1} (X_0)=\{v\in X_0: \|v\|_{X_0} \leq R_1 \}$. By Proposition \ref{Gf1}, $G \circ F$ is a contraction mapping on $B_{R_1} (X_0)$ in the metric of $X$ provided $T$ is sufficiently small. Moreover, it follows that $\mathcal T: B_R(G_0 v_0,Y)  \rightarrow B_R(G_0 v_0,Y)$ is a contraction mapping with respect to the metric of $X$, provided $T$ is small enough.
\hfill $\Box$
\end{proof}

Finally we complete the proof of Theorem \ref{localf1} as follows.

\begin{proof}
We recognize that $B_R(G_0 v_0,Y)$ with respect to $X$-metric is a complete metric space, so by virtue of Lemma \ref{lphe1} and the Contraction Mapping Theorem, we obtain that $\mathcal T$ has a unique fixed point $v\in Y$. Consequently, $v=\mathcal T(v)$ is the unique solution of (\ref{int1}) in the space $Y$, provided $T$ is small enough.

Next, we show that the solution $v\in {\cal C}(I,H^1)$. Indeed, if we introduce the spaces
\begin{align}    \label{Ybar}
\bar Y= \{v\in \bar X, \;  \nabla v\in \bar X \}  \subset {\cal C}(I, H^1), \text{\;\;where\;\;}   \bar X= {\cal C}(I, L^2_{\bf z})  \cap L^4_t  L^4_{
\bf z},
\end{align}
then by Lemma \ref{lgbound} and Proposition \ref{pf1}, we obtain that $G_0 v_0 \in \bar Y$ since $v_0\in H^1$, and $G \circ F(v)\in \bar Y$ since $v\in Y$, and it follows that $v=\mathcal T(v) =  G_0 v_0 + i G \circ F(v) \in \bar Y  \subset {\cal C}(I, H^1)$. By the equation (\ref{dsh1}) we also have
$v_t \in \mathcal C(I, H^{-1})$.

Moreover, we claim that $\nabla \varphi \in {\cal C}(I,W^{4,p})$ for $p>1$. Indeed, since $\varphi_x=B(E(B(|v|^2)))$ and $v\in {\cal C}(I, H^1)$, we obtain that $\varphi_x \in {\cal C}(I, W^{4,p})$, for $p>1$, by using (\ref{inq3}), (\ref{besinq}), and (\ref{posinq}). A similar argument works for $\varphi_y$.

Finally we prove the conservation of the energy $\mathcal N(v)=\|v\|_2^2$. Since $v \in {\cal C}(I,H^1)\cap \, {\cal C}^1(I,H^{-1})$ and $\nabla \varphi \in {\cal C}(I,W^{4,p})$ for $p>1$, we can take the duality pairing of the RDS3 (\ref{dsh3}) with $\bar v$, and it follows that
\begin{align}   \label{ene}
i \langle v_t , \bar v \rangle_{H^{-1}\times H^1} = -   \|\nabla v\|_2^2 - \beta \int_{\mathbb R^2} u |v|^2 dx dy + \rho \int_{\mathbb R^2} \varphi_x |v|^2 dx dy \,.
\end{align}
Notice that the right-hand side of (\ref{ene}) is a real number, thus we take the imaginary part of both sides of (\ref{ene}). Then
\begin{align*}
\text{Re}\, \langle v_t , \bar v \rangle_{H^{-1}\times H^1} =  \frac{1}{2} \frac{d}{dt} \|v\|_2^2 = 0.
\end{align*}
This shows that the energy $\|v\|_2^2$ is invariant in time.

\hfill $\Box$
\end{proof}

\subsection{Continuous dependence on initial data in $H^1$}
This subsection is devoted to prove that the map $v_0 \mapsto (v,\nabla \varphi)$ is continuous from $H^1$ into $\mathcal C(I, H^1) \times \mathcal C(I, W^{4,p})$, for $p>1$, for system (\ref{dsh3}). More precisely, we have the following result.
\begin{theorem}   \label{depend}
Let $v\in \mathcal C(I, H^1)$ and $\nabla \varphi \in \mathcal C(I, W^{4,p})$, for $p>1$, be the solution of the RDS3 system (\ref{dsh3}) with the initial data $v(0)=w\in H^1$. Let $w_n \rightarrow w$ in $H^1$ and $(v_n, \nabla \varphi_n)$ be the solution of (\ref{dsh3}) with the initial value $v_n(0)=w_n$. Then $(v_n,\nabla \varphi_n)$ is defined on $I=[0,T]$, for sufficiently large $n$. Moreover, $v_n \rightarrow v$ in $\mathcal C(I, H^1)$ and $\nabla \varphi_n \rightarrow \nabla \varphi$ in $\mathcal C(I, W^{4,p})$, for $p>1$.
\end{theorem}

\begin{proof}
The proof adopts the idea in \cite{kat}. Let $w\in H^1$. By Theorem \ref{localf1}, there exists a unique solution $(v,\nabla \varphi)$ of the RDS3 system (\ref{dsh3}), on $I=[0,T]$, with the initial value $v(0)=w$, such that $v\in \mathcal C(I, H^1)\cap \mathcal C^1(I, H^{-1})$ and $\nabla \varphi \in \mathcal C(I, W^{4,p})$, for $p>1$. Let $\{w_n\}\subset H^1$ be a sequence of functions such that $w_n \rightarrow w$ in $H^1$. Then there exists a sequence of solutions $(v_n, \nabla \varphi_n)$ to the system (\ref{dsh3}) on $I_n=[0,T_n]$ such that $v_n(0)=w_n$. Notice that $T$ and $T_n$ depends on $\|w\|_{H^1}$ and $\|w_n\|_{H^1}$, respectively, and since $w_n \rightarrow w$ in $H^1$, we see that, for sufficiently large $n$, say $n\geq n_0$, one may take $T_n=T$. That is, $v$ and $\{v_n\}$ all define on $I=[0,T]$, for $n\geq n_0$.

Now, we show that $v_n \rightarrow v$ in $Y \subset \mathcal C(I, H^1)$. Indeed, since $v_n$ and $v$ satisfy (\ref{int1}), one has
\begin{align}    \label{diff}
v_n -  v  =   G_0  (w_n-w) + i [ G \circ F (v_n)  -  G \circ F (v) ].
\end{align}
Take the $X-$norm on both sides of (\ref{diff}) and apply Lemma \ref{lgog1}. We obtain
\begin{align}    \label{hell}
\|v_n-v\|_X  &\leq \|G_0  (w_n-w)\|_X +  \| G \circ F (v_n)  -  G \circ F (v) \|_X    \notag\\
&\leq \gamma \|w_n-w\|_2 + \| G \circ F (v_n)  -  G \circ F (v) \|_X.
\end{align}
We shall estimate the second term on the right-hand side of (\ref{hell}). By the construction of the solutions $v_n$ and $v$, we know that
$v_n \in B_{R}(G_0 w_n, Y)$ and $v \in B_{R}(G_0 w, Y)$. Since $w_n \rightarrow w$ in $H^1$, we see that $v_n \in B_{2R}(G_0 w,Y)$ for sufficiently large $n$. As a result, by (\ref{R1}), there exists $R_1>0$ such that $v_n$, $v\in B_{R_1}(X_0)$. Therefore, by (\ref{li-2}), we have
$$\| G \circ F (v_n)  -  G \circ F (v) \|_X \leq 5 \gamma (\beta+|\rho|) R_1^2 T^{\frac{1}{2}} \|v_n - v\|_X,$$
and along with (\ref{hell}), it follows that
\begin{align}     \label{cha1}
\|v_n-v\|_X  \leq \gamma \|w_n-w\|_2+ 5 \gamma (\beta+|\rho|) R_1^2 T^{\frac{1}{2}} \|v_n - v\|_X.
\end{align}

Next we take the gradient on both sides of (\ref{diff}) and notice that $G_0$ and $G$ are linear operators. One has
\begin{align}  \label{diff-2}
\nabla v_n -  \nabla v  =   G_0 \left( \nabla w_n-  \nabla w \right) + i \left[ G \left( \nabla  F (v_n)  - \nabla F (v) \right) \right].
\end{align}
By taking the $X-$norm on both sides of (\ref{diff-2}) and applying Lemma \ref{lgog1}, it follows that
\begin{align}  \label{hell-2}
\|\nabla v_n -  \nabla v\|_X
&\leq   \| G_0 \left( \nabla w_n-  \nabla w \right) \|_X  +  \|G \left( \nabla  F (v_n)  - \nabla F (v) \right)\|_X \notag\\
&\leq    \gamma \|\nabla w_n-  \nabla w\|_2 + \gamma \|\nabla  F (v_n)  - \nabla F (v)\|_{\frac{4}{3},\frac{4}{3}}.
\end{align}
We shall estimate the second term on the right-hand side of (\ref{hell-2}). Notice
 \begin{align}    \label{f-2}
 \|\nabla  F (v_n)  - \nabla F (v)\|_{\frac{4}{3},\frac{4}{3}} \leq \beta(\tilde I_1+\tilde I_2)+|\rho|(\tilde I_3+\tilde I_4),
 \end{align}
and we claim
 \begin{align}
 & \tilde I_1:=\|\nabla(B(|v_n|^2-|v|^2)v_n)\|_{\frac{4}{3},\frac{4}{3}}\leq C T^{\frac{1}{2}} R_2^2  \|v_n-v\|_Y, \label{I1}\\
 & \tilde I_2:=\|\nabla(B(|v|^2)(v_n-v))\|_{\frac{4}{3},\frac{4}{3}}\leq C T^{\frac{1}{2}} R_2^2  \|v_n-v\|_Y ,\label{I2}\\
 & \tilde I_3:=\|\nabla(B(E(B(|v_n|^2-|v|^2)))v_n)\|_{\frac{4}{3},\frac{4}{3}}\leq C T^{\frac{1}{2}} R_2^2  \|v_n-v\|_Y,\label{I3}\\
 & \tilde I_4:=\|\nabla(B(E(B(|v|^2)))(v_n-v))\|_{\frac{4}{3},\frac{4}{3}}\leq C T^{\frac{1}{2}} R_2^2  \|v_n-v\|_Y ,   \label{I4}
 \end{align}
 for some $R_2 > 0$.
All of the inequalities (\ref{I1})-(\ref{I4}) can be justified similarly, so we solely demonstrate the proof for (\ref{I3}). In fact, by using H\"older's inequality as well as (\ref{L2b}) and (\ref{L2bb}), we deduce
\begin{align}  \label{I31}
\tilde I_3 &\leq \|B(E(B(\nabla |v_n|^2-\nabla |v|^2)))v_n\|_{\frac{4}{3},\frac{4}{3}}+\|B(E(B(|v_n|^2-|v|^2))) \nabla v_n\|_{\frac{4}{3},\frac{4}{3}} \notag\\
&\leq  \left( \int_0^T \|\nabla |v_n|^2  -   \nabla |v|^2 \|_2^{\frac{4}{3}}   \|v_n\|_4^{\frac{4}{3}}  dt  \right)^{\frac{3}{4}}
+  \left( \int_0^T \| |v_n|^2  -  |v|^2 \|_2^{\frac{4}{3}}   \|\nabla v_n\|_4^{\frac{4}{3}}  dt  \right)^{\frac{3}{4}}.
\end{align}
Notice that $|\nabla |v_n|^2-\nabla |v|^2| = |\nabla (v_n \bar v_n)-\nabla (v \bar v)|  \leq 2 |\nabla v_n - \nabla v||v_n|+ 2 |v_n - v| |\nabla v|$. It follows that
\begin{align}   \label{I32}
\|\nabla |v_n|^2-\nabla |v|^2|\|_2
&\leq 2 \| |\nabla v_n - \nabla v||v_n|  \|_2+ 2 \| |v_n - v| |\nabla v|   \|_2 \notag\\
& \leq 2 \|\nabla v_n - \nabla v\|_4  \|v_n\|_4 + 2   \|v_n - v\|_4    \|\nabla v\|_4,
\end{align}
for all $t\in [0,T]$. By (\ref{I31}) and (\ref{I32}), we deduce
\begin{align*}
\tilde I_3  &\leq  CT^{\frac{1}{2}} \big[\|\nabla v_n  -  \nabla  v\|_{4,4} \|v_n\|_{4,\infty}^2
+ \|v_n-v\|_{4,\infty}  \|\nabla v\|_{4,4}   \|v_n\|_{4,\infty}      \notag\\
&\hspace{1 in}+\left(  \|v_n\|_{4,\infty}  +  \|v\|_{4,\infty}  \right)  \|v_n - v\|_{4,\infty}  \|\nabla v_n\|_{4,4} \big]  \notag\\
&\leq C T^{\frac{1}{2}} \left(\|v_n\|_{4,\infty}^2+\|v\|_{4,\infty}^2+\|\nabla v_n\|_{4,4}^2+ \|\nabla v\|_{4,4}^2\right)
\left( \|\nabla v_n - \nabla v\|_{4,4}  +  \|v_n  -  v  \|_{4,\infty}   \right).
\end{align*}
Since $v_n$, $v\in B_{2R}(G_0 w,Y)$ for sufficiently large $n$, there exists $R_2>0$ such that $\|v_n\|_{4,\infty}, \|v\|_{4,\infty}, \|\nabla v_n\|_{4,4}, \|\nabla v\|_{4,4} \leq R_2$ for all $n$. Consequently,
\begin{align*}
\tilde I_3   \leq C T^{\frac{1}{2}} R_2^2  \|v_n-v\|_Y .
\end{align*}
By virtue of (\ref{hell-2})-(\ref{I4}), we obtain
\begin{align} \label{cha2}
\|\nabla v_n -  \nabla v\|_X   \leq  \gamma \|\nabla w_n-  \nabla w\|_2 + C\gamma (\beta+|\rho|)   T^{\frac{1}{2}} R_2^2  \|v_n-v\|_Y .
\end{align}
Combining (\ref{cha1}) and (\ref{cha2}) yields
\begin{align}
\|v_n - v\|_{Y}     \leq     \gamma \|w_n - w\|_{H^1}  + C \gamma (\beta+|\rho|)   T^{\frac{1}{2}} (R_1^2+R_2^2)  \|v_n-v\|_Y.
\end{align}
If $T\leq T^*$, where $T^*$ satisfies $C \gamma (\beta+|\rho|)   (T^*)^{\frac{1}{2}} (R_1^2+R_2^2)=\frac{1}{2} $, then
\begin{align*}
\|v_n - v\|_{Y}    \leq    2 \gamma \|w_n - w\|_{H^1}.
\end{align*}
Since $w_n \rightarrow w$ in $H^1$, we obtain $v_n \rightarrow v$ in $Y  \subset \mathcal C(I, H^1)$. If $T^*$ is shorter than the life span of the solution $v$, the above argument can be iterated. Finally, it is straightforward to deduce that $(\varphi_n)_x=B(E(B(|v_n|^2))) \rightarrow
\varphi_x=B(E(B(|v|^2)))$ in $\mathcal C(I, W^{4,p})$ for $p>1$ by using $v_n \rightarrow v$ in $\mathcal C(I, H^1)$.   \hfill $\Box$
\end{proof}

\subsection{Short-time existence and uniqueness of solutions in $H^2$}   \label{shortH2}
Let $\mathbf z=(x,y)$ and $t\in I=[0,T]$, we introduce the function spaces:
\begin{align}
&Z=\{v\in X:\; v_t \in X, \; \Delta v \in L_t^{\infty} L_{\mathbf z}^2\}, \text{\;\;where\;\;} X=L^{\infty}_t L^2_{\mathbf z} \cap L^4_t L_{\mathbf z}^4,  \label{Z}\\
&\bar Z = \{v\in \bar X: \;v_t \in \bar X, \; \Delta v \in \mathcal C(I, L^2)    \},   \text{\;\;where\;\;}
\bar X=  \mathcal C(I, L^2) \cap L^4_t L^4_{\mathbf z},        \label{barZ}      \\
& Z'=\{f\in L^{\infty}_t L^2_{\mathbf z}:  \;  f_t \in X' \},  \text{\;\;where\;\;}
X'=L_t^1 L_{\mathbf z}^2 + L_t^{\frac{4}{3}} L_{\mathbf z}^{\frac{4}{3}},                \label{Zprime}
\end{align}
with the norm
\begin{align*}
\|v\|_{Z} =   \max\{\|v\|_X, \; \|v_t\|_X, \|\Delta v\|_{2, \infty} \},   \;\;\;\;
\|f\|_{Z'}= \max\{ \|f\|_{2,\infty}, \; \|f_t\|_{X'} \}.
\end{align*}
Recall that $\|v\|=\max\{\|v\|_{2,\infty}, \; \|v\|_{4,4}\}$ and $\|f\|_{X'} = \inf \{\|g\|_{2,1} + \|h\|_{\frac{4}{3},\frac{4}{3}}:  \; f=g+h   \}  $.
Also, note that $v\in Z$ may also be characterized by $v\in L^{\infty}(I, H^2)$ and $v_t \in X$ \cite{kat}.

\begin{theorem}  \label{localH2}
Let $v_0 \in H^2$. Then there exists a unique solution $(v, \nabla \varphi)$ of the RDS3 system (\ref{dsh3}), with the initial value $v(0)=v_0$, on the time interval $I=[0,T]$, for some $T(\|v_0\|_{H^2})>0$, such that $v\in \mathcal C(I, H^2)$, $v_t \in \mathcal C(I, L^2)$, and
$\nabla \varphi \in \mathcal C(I, H^6)$.
\end{theorem}

\begin{proof}
We follow the approach in \cite{kat}. Define the closed ball $B_{R}(Z)=\{v\in Z: \; \|v\|_{Z}  \leq R  \}$. Let $v_0 \in H^2$ and define the set $A$ as
\begin{align*}
A=\{v\in B_{R}(Z):\; v(0)=v_0\}.
\end{align*}
Also, we define the operator $\mathcal T: Z \rightarrow Z$ by
$\mathcal T(v)= G_0 v_0  +  i G \circ F (v)$, where the linear operators $G_0$ and $G$ are defined in (\ref{G0G1}).

We shall show that $\mathcal T(A) \subset A$ provided $R$ is large enough and $T$ is sufficiently small.
Let $v\in A$. Applying Lemma \ref{proZ}, we estimate
\begin{align}  \label{esT}
\|\mathcal T(v)\|_Z  &\leq    \|G_0 v_0\|_Z  +  \|G \circ F (v)\|_Z  \notag\\
&\leq  \|G_0 v_0\|_Z  +  \|G (F (v)-F(v_0) )  \|_Z+      \|G (F(v_0) )  \|_Z   \notag\\
&\leq \gamma \|v_0\|_{H^2}+ (2\gamma+1) \|F (v)-F(v_0)\|_{Z'}+  (2\gamma+1) \|F(v_0)\|_{Z'}.
\end{align}
We shall evaluate the last two terms on the right-hand side of (\ref{esT}). Note that $F(v_0)$ is independent of time,
so by using (\ref{inq4}), (\ref{L2b}) and (\ref{L2bb}), we obtain
\begin{align}       \label{fv0}
&\|F(v_0)\|_{Z'} = \|F(v_0)\|_2 \leq \beta \|B(|v_0|^2)\|_2  \|v_0\|_{\infty} + |\rho| \|B(E(B(|v_0^2|)))\|_2 \|v_0\|_{\infty}    \notag\\
& \leq C(\beta + |\rho| )\||v_0|^2\|_2  \|v_0\|_{H^2}   =   C(\beta + |\rho|) \|v_0\|_4^2  \|v_0\|_{H^2}   \leq C(\beta+|\rho|) \|v_0\|_{H^2}^3.
\end{align}
Next, we estimate $\|F (v)-F(v_0)\|_{Z'}$. Indeed, by Lemma 3.3 in \cite{kat}, we have
\begin{align}  \label{katot}
\|v(t)-v(s)\|_{2p}  \leq C |t-s|^{\theta} \|v\|_Z,   \text{\;\;for\;\;} k=\frac{p-1}{p}, \;\;\; \theta = 1-\frac{k}{2}.
\end{align}
By using H\"older inequality, along with (\ref{besinq}), (\ref{posinq}) and (\ref{katot}), we evaluate
\begin{align*}
&\|F(v)-F(v_0)\|_2  \notag\\
&\leq  \beta \left(\|B(|v|^2-|v_0|^2)\|_3  \|v\|_6   +   \|B(|v_0|^2)\|_3  \|v-v_0\|_6       \right)  \notag\\
&\hspace{0.5 in} +|\rho| \left(   \|B(E(B(|v|^2-|v_0|^2)))\|_3     \|v\|_6  +     \|B(E(B(|v_0|^2)))\|_3     \|v-v_0\|_6    \right)      \notag\\
&\leq C\beta \left(\||v|^2-|v_0|^2\|_3  \|v\|_6   +   \||v_0|^2\|_3  \|v-v_0\|_6       \right)    \notag\\
&\hspace{0.5 in} + C|\rho| \left(   \||v|^2-|v_0|^2\|_3     \|v\|_6  +     \||v_0|^2\|_3     \|v-v_0\|_6    \right)   \notag\\
&\leq  C\beta \left(\||v|+|v_0|\|_6 \|v-v_0\|_6   \|v\|_6   +   \|v_0\|_6^2  \|v-v_0\|_6       \right)    \notag\\
&\hspace{0.5 in} +C|\rho| \left(   \||v|+|v_0|\|_6  \|v-v_0\|_6   \|v\|_6  +     \|v_0\|_6^2     \|v-v_0\|_6    \right)   \notag\\
&\leq C(\beta+|\rho|)\left(\|v\|_6^2+\|v_0\|_6^2 \right) \|v-v_0\|_6  \notag\\
&\leq C(\beta+|\rho|)\left(\|v\|_6^2+\|v_0\|_6^2 \right)  t^{\frac{2}{3}} \|v\|_{Z},    \;\;\;\; \text{for all\;\;}t\in [0,T].
\end{align*}
It follows that
\begin{align}   \label{Fes1}
\|F(v)-F(v_0)\|_{2,\infty} &\leq  C(\beta+|\rho|)\left(\|v\|_Z^2+\|v_0\|_{H^2}^2 \right)  T^{\frac{2}{3}} \|v\|_{Z}  \notag\\
& \leq C(\beta+|\rho|)\left(R^2+\|v_0\|_{H^2}^2 \right)  T^{\frac{2}{3}} R.
\end{align}

Recall $X'=L^1_t L^2_{\mathbf z} + L^{\frac{4}{3}}_t L^{\frac{4}{3}}_{\mathbf z}$, with the norm $\|f\|_{X'}=\inf\{\|g\|_{2,1}+\|h\|_{\frac{4}{3},\frac{4}{3}}:  \, f=g+h \} $. Thus $\|\partial_t (F(v)-F(v_0))\|_{X'} \leq \|\partial_t (F(v)-F(v_0))\|_{\frac{4}{3},\frac{4}{3}}= \|\partial_t F(v)\|_{\frac{4}{3},\frac{4}{3}}$.
We denote by $\tau_s$ the shift of time by $s\in \mathbb R$, i.e., $\tau_s v (t)=v(t+s)$. Also, we denote the identity operator by Id, then by applying (\ref{f1b}), we deduce
\begin{align*}
\|(\tau_s-\text{Id})F(v)\|_{\frac{4}{3},\frac{4}{3}}
=\|F(\tau_s v)-F(v)\|_{\frac{4}{3},\frac{4}{3}} \leq 5 (\beta+|\rho|) \|v\|_{X_0}^2 T^{\frac{1}{2}} \|(\tau_s - \text{Id})v\|_{4,4}.
\end{align*}
Dividing by $|s|$ and letting $s\rightarrow 0$, one has
\begin{align*}
\|\partial_t F(v)\|_{\frac{4}{3},\frac{4}{3}} \leq 5 (\beta+|\rho|) \|v\|_{X_0}^2 T^{\frac{1}{2}} \| v_t \|_{4,4}.
\end{align*}
This shows that
\begin{align}  \label{Fes2}
\|\partial_t (F(v)-F(v_0)) \|_{X'}   \leq 5 (\beta+|\rho|) \|v\|_{X_0}^2 T^{\frac{1}{2}} \| v_t \|_X \leq C (\beta+|\rho|)  T^{\frac{1}{2}} R^3.
\end{align}

Combining (\ref{Fes1}) and (\ref{Fes2}) yields
\begin{align}    \label{Fes3}
\|F(v)-F(v_0)\|_{Z'}  \leq C (\beta+|\rho|) [\left(R^2+\|v_0\|_{H^2}^2 \right)  T^{\frac{2}{3}} R+ T^{\frac{1}{2}} R^3].
\end{align}

By (\ref{esT}) , (\ref{fv0}) and (\ref{Fes3}), we obtain
\begin{align*}
\|\mathcal T(v)\|_Z   \leq   \gamma \|v_0\|_{H^2}+ (2\gamma+1) C (\beta+|\rho|) [\left(R^2+\|v_0\|_{H^2}^2 \right)  T^{\frac{2}{3}} R
+ T^{\frac{1}{2}} R^3 + \|v_0\|_{H^2}^3 ].
\end{align*}
If we let $R > \gamma \|v_0\|_{H^2}+ (2\gamma+1) C (\beta+|\rho|) \|v_0\|_{H^2}^3 $, and choose $T$ sufficiently small, then the above estimate implies
$\|\mathcal T(v)\|_Z < R$. Also, notice that $\mathcal T(v)(0)=v_0$. So $\mathcal T(A)\subset A$.

Next, we let $v$, $w\in A$, and using Lemma \ref{lgog1} and (\ref{f1b}), we deduce
\begin{align*}
&\|\mathcal T(v)-\mathcal T(w)\|_X  = \|G(F(v)-F(w))\|_X   \leq \gamma \|F(v)-F(w)\|_{\frac{4}{3},\frac{4}{3}}   \notag\\
&\leq 5 \gamma (\beta+|\rho|) \max\{ \|v\|_{X_0}^2, \|w\|_{X_0}^2 \} T^{\frac{1}{2}} \|v-w\|_{4,4}
\leq C \gamma (\beta+|\rho|) R^2 T^{\frac{1}{2}} \|v-w\|_X.
\end{align*}
Consequently, $\mathcal T: A \rightarrow A$ is a contraction mapping in the norm of $X$, provided $T$ is sufficiently small. It follows that $\mathcal T$ has a unique fixed point in the set $A$ with respect to the metric of $X$ by virtue of the contraction mapping theorem, i.e, there exists $v\in A$ such that $v=\mathcal T(v)=G_0 v_0+ i G \circ F(v)\in \bar Z$, due to Lemma \ref{proZ}. Therefore, $v\in \mathcal C(I, H^2)$ and $v_t \in \mathcal C(I, L^2)$. Finally, $v\in \mathcal C(I,H^2)$ implies $|v|^2\in \mathcal C(I,H^2)$ since the spatial dimension is two, and thus by (\ref{besinq}) and (\ref{L2bb}), one has $\phi_x=B(E(B(|v^2|)))\in  \mathcal C(I,H^6)$.   \hfill $\Box$
\end{proof}

\bigskip

\subsection{Conservation of the Hamiltonian}
\begin{theorem}
 Assume the initial datum $v_0\in H^1$. Let $v\in \mathcal C(I, H^1) \cap \mathcal C^1(I, H^{-1})$ with $\nabla \phi \in \mathcal C(I, W^{4,p})$, $p>1$, be the solution of RDS3 system (\ref{dsh3}). Then the Hamiltonian $$\mathcal H_3(v)=\int_{\mathbb{R}^2}\left[|\nabla v|^2 -  \frac{\beta}{2} u|v|^2 + \frac{\rho}{2} \left(\psi_x^2+\nu \psi_y^2   \right)\right]\,dxdy $$ is conserved in time.
\end{theorem}

\begin{proof}
First, we assume $v_0\in H^2$, then by Theorem \ref{localH2}, the RDS3 system (\ref{dsh3}) has a unique solution $v\in \mathcal C(I, H^2)$ with $v_t \in \mathcal C(I, L^2)$ and $\nabla \phi \in \mathcal C(I,H^6)$. Therefore it is legitimate to take the inner product of the equation (\ref{dsh3}) with $\bar v_t$ to obtain
\begin{align}   \label{H-1}
i \int_{\mathbb R^2}|v_t|^2 dx dy   -   \int_{\mathbb R^2}  \nabla v \cdot \nabla \bar v_t  dx dy
 + \beta   \int_{\mathbb R^2}  u    v  \bar v_t dx dy  -\rho  \int_{\mathbb R^2}  \varphi_{x}v   \bar v_t  dx dy  =0.
\end{align}
Now we take the real part of each term in the above equality. Clearly,
\begin{align}   \label{H-2}
\text{Re} \int_{\mathbb R^2}  \nabla v \cdot \nabla \bar v_t  dx dy
=\frac{1}{2}\frac{d}{dt} \int_{\mathbb R^2}  |\nabla v|^2 dx dy.
\end{align}

Moreover, since $u- \alpha^2 \Delta u = |v|^2$, we see that $u$ is real-valued, and thus
\begin{align}     \label{H-3}
&\text{Re} \int_{\mathbb R^2}  u    v  \bar v_t dx dy = \frac{1}{2}  \int_{\mathbb R^2} u \partial_t (|v|^2) dx dy
=\frac{1}{2}  \int_{\mathbb R^2} u (u_t-\alpha^2 \Delta u_t) dx dy  \notag\\
&=\frac{1}{4}\frac{d}{dt} \int_{\mathbb R^2} \left( u^2 + \alpha^2 |\nabla u|^2   \right) dx dy
=\frac{1}{4}\frac{d}{dt} \int_{\mathbb R^2} \left( u^2 - \alpha^2 u \Delta u  \right) dx dy
=\frac{1}{4}\frac{d}{dt} \int_{\mathbb R^2} u |v|^2 dx dy.
\end{align}

Recall from system (\ref{dsh3}) that $\varphi-\alpha^2 \Delta \varphi=\psi$ and $\Delta_{\nu} \psi=u_x$. Also, since $\varphi$ is real-valued, we deduce that
\begin{align}    \label{H-4}
&\text{Re} \int_{\mathbb R^2}  \varphi_{x}v   \bar v_t  dx dy = \frac{1}{2}\int_{\mathbb R^2}  \varphi_{x}  \partial_t (|v|^2)  dx dy
=\frac{1}{2}\int_{\mathbb R^2}  \varphi_x  \partial_t (u-\alpha^2 \Delta u)  dx dy  \notag\\
&=-\frac{1}{2}  \int_{\mathbb R^2} (\varphi-\alpha^2 \Delta \varphi) u_{xt} dx dy
=-\frac{1}{2}  \int_{\mathbb R^2} \psi \Delta_{\nu} \psi_t dx dy = \frac{1}{4} \frac{d}{dt} \int_{\mathbb R^2} \left(\psi_{x}^2+\nu \psi_y^2 \right) dx dy.
\end{align}

By (\ref{H-1})-(\ref{H-4}), we conclude that
\begin{align*}
\frac{d}{dt} \left(\int_{\mathbb R^2}  |\nabla v|^2 dx dy- \frac{\beta}{2} \int_{\mathbb R^2} u |v|^2 dx dy
+ \frac{\rho}{2} \int_{\mathbb R^2} \left(\psi_{x}^2+\nu \psi_y^2 \right) dx dy \right)=0\, ,
\end{align*}
i.e. $\frac{d}{dt} \mathcal H_3(v)=0$.
This shows $\mathcal H_3(v)$ is invariant in time provided $v\in \mathcal C(I, H^2)$ with $v_t \in \mathcal C(I, L^2)$.

Next, we consider the general initial data: $v_0 \in H^1$. Take a sequence of functions $\{w_n\}\subset H^2$ such that $w_n\rightarrow v_0$ in $H^1$. Then by Theorem \ref{localH2}, there exists a sequence of solutions $\{v_n\}$ of (\ref{dsh3}) on $I_n=[0,T_n]$, with the initial value $v_n(0)=w_n$, such that $v_n\in \mathcal C(I_n, H^2)$, $\partial_t v_n \in \mathcal C(I_n, L^2)$ and $\nabla \varphi_n \in \mathcal C(I_n, H^6)$. By the above result, we know that $\mathcal H_3(v_n)$ is conserved in time. Moreover, by Theorem \ref{depend}, we see that, for sufficiently large $n$, we have $v_n$ is defined on $I=[0,T]$, such that $v_n \rightarrow v$ in $\mathcal C(I, H^1)$, $\nabla \varphi_n \rightarrow \nabla \varphi$ in $\mathcal C(I, W^{4,p})$.
If follows that $u_n\rightarrow u$ in $\mathcal C(I,H^3)$ and $\nabla \psi_n \rightarrow \nabla \psi$ in $\mathcal C(I, W^{2,p})$, for $p>1$. As a result, we conclude that $\mathcal H_3(v_n)\rightarrow \mathcal H_3(v)$ on $[0,T]$, and thus $\mathcal H_3(v)$ is conserved in time.
\end{proof}

\bigskip

\subsection{The extension to global solutions in $H^1$}
In the proof of the short-time existence and uniqueness theorem for the RDS3 system (\ref{dsh3}) in section \ref{short}, we have produced the estimates that are necessary for implementing the contraction mapping argument, on the time interval $[0,T]$, where $T$ is taken to be small enough depending on the initial data. The solution of the RDS3 (\ref{dsh3}) established in Theorem  \ref{localf1} can be extended to a maximal interval of existence $[0, T_{\max})$, where $T_{\max}$ might be finite or infinite. In this section, we establish the global existence of solutions to the Cauchy problem (\ref{dsh3}), by using the conservation of the energy and the Hamiltonian. To do this, we focus attention on the maximal interval of existence $[0,T_{\max})$. If $T_{\max}=\infty$, then the solutions exist globally in time. On the other hand, if $T_{\max}<\infty$, then one has
\begin{equation}
\limsup\limits_{t\rightarrow T^-_{\max}}\|v(t)\|_{H^1}=\infty, \label{limsup}
\end{equation}
otherwise, one can extend the solution, beyond $T_{\max}$, which contradicts the fact that $T_{\max}$ is the maximal time of the existence. This argument is used to prove the global existence theorem in this section, hence we assume by contradiction that $T_{\max}<\infty$ and then show that (\ref{limsup}) does not hold, which implies that $T_{\max}=\infty$.

Now we present the proof for the extension to global solutions for the system RDS3 (\ref{dsh3}), which completes the proof of the global well-posedness of (\ref{dsh3}) stated in Theorem \ref{global3}.

 \begin{proof} Let $[0,T_{\max})$ be the maximal interval of existence of the solution established in Theorem \ref{localf1}. We assume $T_{\max}<\infty$.
It has been shown that the energy ${\cal N}(v)=\|v\|_2^2$, and the Hamiltonian
\begin{align}   \label{ham3}
{\cal H}_3 (v) = \int_{\mathbb{R}^2}  \left[ |\nabla v|^2- \frac{\beta}{2} u |v|^2
 + \frac{\rho}{2} \left(\psi_x^2+\nu \psi_y^2 \right)\right]  dxdy  \, ,
\end{align}
remains constant for all $t \in [0,T_{\max})$. We aim to derive a uniform bound of $\|v\|_{H^1}$ by using the conservation of the energy and the Hamiltonian. Indeed, it can be readily seen from (\ref{ham3}) that
\begin{align} \label{b-g}
\|\nabla v\|^2_2 = {\cal H}_3 (v) + \frac{\beta}{2} \int_{\mathbb{R}^2} u|v|^2\,dxdy - \frac{\rho}{2}  \left(\|\psi_x\|^2_2+\nu\|\psi_y\|^2_2  \right).
\end{align}
Recall that $u-\alpha \Delta u=|v|^2$, i.e., $u=B(|v|^2)$. By using (\ref{inq1}), (\ref{inq4}) and (\ref{besinq}), we estimate
\begin{align}  \label{qub}
 &\int_{\mathbb{R}^2} u|v|^2\,dxdy \leq  \|u\|_{L^{\infty}} \|v\|_2^2  \leq C\|u\|_{H^2} \|v\|_2^2 = C\|B(|v|^2)\|_{H^2} \|v\|_2^2  \notag\\
&\;\;\leq C_{\alpha} \| |v|^2 \|_2 \|v\|_2^2 =  C_{\alpha} \| v \|_4^2 \|v\|_2^2 \leq C_{\alpha} \|v\|_{H^1} \|v\|_2^3 \,,
\end{align}
where $C_\alpha \sim  1/\alpha^2$, as $\alpha \rightarrow 0^+$.

By system (\ref{dsh3}) one has $\Delta_{\nu}  \psi=u_x$, it follows that $\psi_x=E(u)$ where the operator $E$ is defined in (\ref{eop}).
Since $u=B(|v|^2)$, we obtain $\psi_x=E(B(|v|^2))$. We estimate $\|\psi_x\|_2$ in the frequency space:
\begin{align}  \label{thi1}
&\|\psi_x\|_2^2  =  \|E(B(|v|^2))\|_2^2  = \int_{\mathbb R^2}  \frac{\xi_1^4}{(\xi_1^2 + \nu \xi_2^2)^2}  \frac{1}{(1+\alpha^2 |\xi|^2)^2} |\widehat{|v|^2}(\xi)|^2 \; d\xi_1 d\xi_2 \notag\\
&\leq  \int_{\mathbb R^2}   \frac{1}{(1+\alpha^2 |\xi|^2)^2} |\widehat{|v|^2}(\xi)|^2 \; d\xi_1 d\xi_2
\leq \|v\|_2^4  \int_{\mathbb R^2}  \frac{1}{(1+\alpha^2 |\xi|^2)^2} \; d\xi_1 d\xi_2
=\frac{C}{\alpha^2}  \|v\|_2^4\, ,
\end{align}
where we have used above the convolution theorem and Young inequality for convolution to obtain
$|\widehat{|v|^2}(\xi)| = |(\widehat{v\cdot \bar v}) (\xi)|
= |(\hat v \ast \hat{\bar v})(\xi)| \leq \|v\|_2^2 $, for every $\xi\in \mathbb R^2$. Similarly,
\begin{align}     \label{thi2}
\|\psi_y\|_2^2  \leq \frac{C(\nu)}{\alpha^2}  \|v\|_2^4.
\end{align}

By (\ref{b-g}), (\ref{qub}), (\ref{thi1}) and (\ref{thi2}), one has
\begin{align*}
\|\nabla v(t)\|^2_2
&\leq  {\cal H}_3(v(t)) + \frac{\beta}{2} C_{\alpha}\|v(t)\|_{H^1} \|v(t)\|_2^3+\frac{|\rho|}{\alpha^2} C_{\nu}  \|v(t)\|_2^4      \notag\\
&\leq {\cal H}_3(v_0)+ \frac{1}{2}\|v(t)\|^2_{H^1} + \frac{\beta^2}{8}  C_{\alpha}^2 \|v_0\|_2^6
+   \frac{|\rho|}{\alpha^2} C_{\nu} \|v_0\|_2^4   \, ,
\end{align*}
due to the Young inequality as well as the conservation of the energy $\|v\|_2$ and the Hamiltonian ${\cal H}_3(v)$. Since $\|v\|^2_{H^1}=\|v\|_2^2+\|\nabla v\|_2^2$, it follows that
\begin{align*}
\|\nabla v(t)\|_2^2  & \leq 2 {\cal H}_3(v_0) + \|v_0\|^2_2 + \frac{\beta^2}{4} C_{\alpha}^2 \|v_0\|_2^6  + \frac{2|\rho|}{\alpha^2} C_{\nu} \|v_0\|_2^4,
\end{align*}
for all $t\in [0,T_{max})$. Consequently,
 \[ \limsup\limits_{t\rightarrow T^-_{\max}}\|v(t)\|_{H^1}<\infty   \, ,\]
which contradicts (\ref{limsup}), and hence the solution exists globally in time. \hfill $\Box$
 \end{proof}

\bigskip

\section{Modulation theory}   \label{modulation}
Modulation theory is introduced in order to explain the role of the regularization, through perturbation of a system that develops a singularity, in preventing singularity formation of the original system. The intention of this theory is that  the profiles of the perturbed system's solutions are asymptotic to some rescaled profiles of the original system's solutions near the singularity. By this approach, a perturbed system can be reduced into a simpler system of ordinary differential equations that do not depend on the spatial variables, and are easier to analyze both analytically and numerically. Hence, in this section, we will apply modulation theory to the RDS3 system (\ref{dsh3}) by following the ideas in \cite{fib,pap,sul} (see also \cite{Cao1}) for the purpose of observing the prevention mechanism of the singularities.

 First, we review some main results on an asymptotic construction of blow-up solutions for the DSE presented in \cite{pap,sul}.
 It is convenient to rewrite the DSE (\ref{ds}) in the terms of the amplitude $v$ and the longitudinal velocity $u_1=\phi_x$ in the form
 \begin{align}
 \begin{cases}
 i v_t + \Delta v+\beta |v|^2 v - \rho u_1 v=0   \\
 \Delta_{\nu}u_1=(|v|^2)_{xx}.      \label{ds2}
 \end{cases}
 \end{align}
 It is shown in \cite{pap,sul} that blow-up solutions of system (\ref{ds2}) have the following asymptotic form near the singularity:
\begin{align}
\begin{cases}
    v({x,y},t)\approx {\frac{1}{L(t)}} e^{i (\tau(t)- a(t) \frac{|\eta|^2}{4}) } P\left(|\eta|, b(t)\right), \\
    u_1({x,y},t)\approx -{\frac{1}{L^2(t)}}(-\Delta_{\nu})^{-1}(|P|^2)_{\eta_1\eta_1}  , \label{as}
\end{cases}
\end{align}
where $\eta=(\eta_1,\eta_2)=(\frac{x}{L}, \frac{y}{L})$, $\tau_t=L^{-2}$, $a=-L_t L$ and $b=a^2 + a_{\tau} \approx a^2$, which satisfies $b_\tau \sim -e^{-\frac{\pi}{\sqrt{b}}}$. Also, to leading order at the limit as $\tau \rightarrow \infty$, one has $b \sim \frac{1}{(\ln \tau)^2}$. The steady system of (\ref{ds2}) reads (see \cite{pap,sul})
 \begin{align}   \label{ds3}
 \begin{cases}
 \Delta P-P+\frac{b}{4} |\eta|^2 P +i  \sqrt{b} \left(\frac{1}{p}-1\right)  P+\beta  |P|^{2p} P - \rho PQ =0 \\
 \Delta_{\nu}  Q=(|P|^{2p})_{\eta_1\eta_1}
 \end{cases}
 \end{align}
 where $p>1$. Note that in the limit, $b \rightarrow0$ and $p\rightarrow1$, as $\tau \rightarrow \infty$, and $(P,Q)$ tends to $(S,X)$, a solution of
\begin{align}  \label{SX}
\begin{cases}
\Delta  S-S+\beta S^3-\rho SX=0, \\
\Delta_{\nu} X= (S^{2})_{\eta_1\eta_1},
\end{cases}
\end{align}
 with zero boundary conditions at infinity. It is also obtained in  \cite{pap,sul} that, as $b$ tends to 0,
 one has $1-\frac{1}{p} \sim \frac{1}{\sqrt{b}}e^{-\frac{\pi}{\sqrt{b}}}$ and the scaling factor $L(t)$ approaches zero, in the case of self-focusing of the original system, like $L(t) \sim (t^*-t)^{\frac{1}{2}} \left(\ln  \ln   \frac{1}{t^* - t}  \right)^{-\frac{1}{2}} .$

 Observe that the singularity in the original system (\ref{ds2}) is manifested by the fact that $L(t)$ tends to 0, as $t\rightarrow t^*$. Thus, our goal is now to show how the regularization mechanism prevents $L(t)$ from collapsing to zero.

We adopt a similar strategy as in \cite{pap,sul}. The following arguments are formal and have not been placed on the level of mathematical rigor. For small values of the parameter $\alpha$, the RDS3 system (\ref{dsh3}) can be regarded as a perturbation of the DSE (\ref{ds}). To see this, we define $$\Phi=\varphi_x,\hskip0.2cm   \Psi=\psi_x,$$
for the sake of convenience. Then equation (\ref{dsh3}) becomes
 \begin{align} \label{dsh3p1}
 \begin{cases}
  i v_t+\Delta v+\beta u v-\rho \Phi  v=0,\hspace{0.5cm} \Delta_\nu \Psi=u_{xx}, \\
  \hspace{0.5cm} u-\alpha^2 \Delta u=|v|^2, \hspace{0.5cm} \Phi-\alpha^2 \Delta \Phi=\Psi ,
 \end{cases}
 \end{align}
 and $u$ and $\Phi$ can be formally expanded in leading order $\alpha^2$ as:
 \begin{align*}
 &u = |v|^2+\alpha^2 \Delta u = |v|^2 + \alpha^2 \Delta (|v|^2+\alpha^2 \Delta u)=|v|^2+\alpha^2 \Delta |v|^2+{\cal O}(\alpha^4) \\
 &\Phi=\Psi+\alpha^2\Delta \Phi=\Psi+\alpha^2\Delta (  \Psi+\alpha^2\Delta \Phi  )=
 \Psi+\alpha^2 \Delta \Psi+{\cal O}(\alpha^4) .
 \end{align*}
Thus we can rewrite equation (\ref{dsh3p1}) to the leading order of $\alpha^2$ as
 \begin{align} \label{dsh3p2}
 \begin{cases}
  i v_t+ \Delta v+\beta |v|^2 v-\rho \Psi v+ \alpha^2 \left(\beta v \Delta |v|^2 - \rho v \Delta \Psi \right) =0      \\
 \Delta_{\nu} \Psi = (|v|^2)_{xx}+\alpha^2 \Delta (|v|^2)_{xx} .
 \end{cases}
 \end{align}

The numerical simulations \cite{pap} suggest that the blow-up of the DSE (\ref{ds}) is very similar to that of the critical NLS (\ref{nls})
and the typical scales remain comparable in the $x$ and $y$ directions. Therefore we choose to use the same scaling factor $L(t)$ in both directions.
As in \cite{pap,sul}, we define
\begin{align*}
&\xi_1=\frac{x}{L(t)};  \;\;\;\;\;      \xi_2=\frac{y}{L(t)}   ;     \;\;\;\;\;  \tau=\int_0^t  \frac{1}{L^2(s)} ds ; \\
&  U(\xi_1,\xi_2,\tau) = L(t) v(x,y,t);     \;\;\;\;\; W(\xi_1,\xi_2,\tau) =  L^2(t)  \Psi(x,y,t).
\end{align*}
Since $U$ and $W$ depend on the new variables $\xi_1$, $\xi_2$ and $\tau$, in what follows we denote
\begin{align*}
\nabla = (\partial_{\xi_1}, \partial_{\xi_2}), \;\;\;\;\;  \Delta= \partial_{\xi_1\xi_1} + \partial_{\xi_2 \xi_2}, \;\;\;\;\; \Delta_{\nu}= \partial_{\xi_1 \xi_1} + \nu \partial_{\xi_2 \xi_2}.
\end{align*}

Notice that
$v_t = \partial_t \left[  \frac{U(\xi_1,\xi_2,\tau)}{L(t)}   \right]=  \frac{1}{L^3} \left[  U_{\tau}+a (U+\xi \cdot \nabla U )  \right],$
where $a=-L_t L$ and $\xi=(\xi_1,\xi_2)$. Then equation (\ref{dsh3p2}) can be written as
\begin{align*}
\begin{cases}
i U_{\tau} + i a(U+\xi \cdot \nabla U)+\Delta U+\beta |U|^2 U-\rho W U
+\epsilon\left(\beta U \Delta |U|^2   -\rho U \Delta W  \right) =0  \\
\Delta_{\nu} W = (|U|^2)_{\xi_1 \xi_1}+  \epsilon \Delta (|U|^2)_{\xi_1  \xi_1}
\end{cases}
\end{align*}
where $\epsilon=\frac{\alpha^2}{L^2}$. Inspired by (\ref{as}) we set
\begin{align*}
U(\xi,\tau)=e^{ i(\tau - a \frac{|\xi|^2}{4})} V(\xi,\tau),
\end{align*}
and let $b = a_{\tau}+a^2$. Therefore
\begin{align}    \label{rescaled}
\begin{cases}
i V_{\tau} + \Delta V -V +  \frac{b}{4}|\xi|^2  V    + \beta |V|^2 V  - \rho W V +
\epsilon (\beta V \Delta |V|^2 - \rho V \Delta W )=0    \\
\Delta_{\nu} W = (|V|^2)_{\xi_1 \xi_1}+  \epsilon \Delta (|V|^2)_{\xi_1  \xi_1}.
\end{cases}
\end{align}

We observe that, on one hand, equation (\ref{rescaled}) becomes the rescaled form of the RDS1 system (\ref{dsh1}) if the terms $-\epsilon \rho V \Delta W$
and $\epsilon \Delta (|V|^2)_{\xi_1  \xi_1}$ are neglected. On the other hand, if the term $\epsilon \beta V \Delta |V|^2$ is omitted from (\ref{rescaled}), the equation becomes the rescaled form of the RDS2 system (\ref{dsh2}). Therefore, the argument in this section can also be applied to the RDS1 and RDS2 systems in a straightforward manner.

Analogously to \cite{pap, sul}, we formally modulate the degree of the nonlinearity, and introduce the steady state system (similar to (\ref{ds3}))
\begin{align} \label{dsh3p5}
 \begin{cases}
     \Delta V^0-V^0 + \frac{b}{4}|\xi|^2 V^0 + \beta |V^0|^{2p} V^0 -\rho W^0 V^0
       +i \sqrt{b} \left(\frac{1}{p}-1\right) V^0 \\
    \hspace{2cm}   +\epsilon(\beta V^0 \Delta |V^0|^{2p} - \rho  V^0 \Delta W^0)=0, \\
   \Delta_{\nu} W^0=(|V^0|^{2p})_{\xi_1 \xi_1} + \epsilon \Delta (|V^0|^{2p})_{\xi_1\xi_1},
 \end{cases}
 \end{align}
with $p>1$ and $b>0$, where $V^0(|\xi|,b(\tau),\epsilon(\tau))$ and  $W^0(|\xi|,b(\tau),\epsilon(\tau))$ are quasi-steady in $\tau$.

At this stage, we expand $V^0$ and $W^0$ with respect to small values of $b$ and $\epsilon$:
 \begin{align}  \label{app}
 \begin{cases}
 V^0 = S(|\xi|) + b(\tau) G(|\xi|)+\epsilon(\tau) H(|\xi|)+{\cal O}(b^2,\epsilon^2)\\
 W^0 = X(|\xi|) + b(\tau) Y(|\xi|)+\epsilon(\tau) Z(|\xi|)+{\cal O}(b^2,\epsilon^2).
 \end{cases}
 \end{align}
 We consider the condition $p(b(\tau))\rightarrow 1^+$, as $\tau\rightarrow \infty$, then the equations for $(S,X)$  are given by
  \begin{align}    \label{SX}
  \begin{cases}
  \Delta S -S +\beta S^3 -\rho S X = 0\\
  \Delta_{\nu} X-(S^2)_{\xi_1 \xi_1} = 0
  \end{cases}
  \end{align}
 and the equations for $(G,Y)$ are
 \begin{align}   \label{GY}
  \begin{cases}
  \Delta G-G+3 \beta GS^2-\rho (SY+GX)= -\frac{1}{4}|\xi|^2 S \\
  \Delta_{\nu} Y-2(GS)_{\xi_1\xi_1}=0
  \end{cases}
 \end{align}
 with zero boundary conditions at infinity.

 Notice that (\ref{SX}) is a system of nonlinear PDEs, which is essentially identical to the system (\ref{groundds}), whose solutions are ground states (standing waves) of DSE, and the existence, regularity, and asymptotics of the ground states have been studied in \cite{cip}. On the other hand, (\ref{GY}) is a system of linear equations, and due to the Fredholm alternative, (\ref{GY}) is solvable provided the vector determined by the right-hand side of the system is orthogonal to the kernel of the adjoint of the operator arising in the left-hand side. In particular, the vector $(-\frac{1}{4}|\xi|^2 S,\; 0)$ needs to be orthogonal to the solution set of the equation
 \begin{align}   \label{GY-dual}
  \begin{cases}
  \Delta \tilde G- \tilde G+3 \beta \tilde G S^2-\rho \tilde G X  - 2 S \tilde Y_{\xi_1\xi_1} = 0 \\
  \Delta_{\nu} \tilde Y - \rho S \tilde G = 0.
  \end{cases}
 \end{align}
By virtue of (\ref{SX}), the solution set of (\ref{GY-dual}) is spanned by
  \begin{equation}   \label{GYdualsol}
 \left\{\left(
         \begin{array}{c}
           S_{\xi_1} \\
           \frac{\rho}{2}X_1 \\
         \end{array}
       \right),
  \left(
         \begin{array}{c}
            S_{\xi_2}  \\
            \frac{\rho}{2}X_2 \\
         \end{array}
       \right)\right\}
 \end{equation}
 where $(X_1)_{\xi_1}=X$ and $(X_2)_{\xi_1\xi_1}=X_{\xi_2}$. As a result, the solvability condition of system (\ref{GY}) is
 $\int_{\mathbb{R}^2}|\xi|^2 S S_{\xi_j}\,d\xi_1d\xi_2=0$, $j=1,2$, that is,
 $$\int_{\mathbb{R}^2}\xi_j S^2\,d\xi_1d\xi_2=0, \;\;\; j=1,2.$$ This condition is satisfied since $S$ is symmetric with respect to the variables $\xi_1$ and $\xi_2$, which is confirmed by numerical simulations \cite{pap,sul}.

Moreover, the equations for $(H,Z)$ are
\begin{align}   \label{HZ}
 \begin{cases}
 \Delta H-H+3\beta H S^2-\rho (SZ + HX) = -\beta S \Delta (S^2) + \rho S \Delta X  \\
 \Delta_{\nu} Z- 2(SH)_{\xi_1\xi_1} = \Delta (S^2)_{\xi_1\xi_1},
  \end{cases}
 \end{align}
 with zero boundary conditions at infinity. Existence of solutions for (\ref{HZ}) requires that the right-hand side of (\ref{HZ}) be orthogonal to the
 kernel of the adjoint of the operator arising in the left-hand side, which is also spanned by the vectors given in (\ref{GYdualsol}).
  The solvability condition of system (\ref{HZ}) thus reads
 \[
 \int_{\mathbb{R}^2}  \left[ -\beta  \Delta (S^2)  \partial_{\xi_j}(S^2)+\rho \partial_{\xi_j} (S^2) \Delta X +\rho \Delta (S^2) X_{\xi_j} \right] \; d\xi_1d\xi_2=0,
 \]
for $j=1,2$, which can be reduced to,
\begin{align*}
\int_{\mathbb{R}^2} \Delta (S^2)  \partial_{\xi_j}(S^2)\; d\xi_1 d\xi_2=0,
\end{align*}
which is valid since $S$ is symmetric with respect to $\xi_1$ and $\xi_2$.

Next, we consider the unsteady problem (\ref{rescaled}). Let $V=V^0+V^1$ and $W=W^0+W^1$. Using (\ref{rescaled}) and (\ref{dsh3p5}), we obtain a system for the remainder $V^1$ and $W^1$:
 \begin{align*}
 \begin{cases}
  \Delta V^1 - V^1 +  \frac{b}{4}|\xi|^2 V^1   +  \beta \left( |V^0+V^1|^2(V^0+V^1)- |V^0|^{2p}V^0 \right)
       \\
  ~~ -\rho (W^1 V^0 + W^0 V^1 + W^1 V^1 ) + \epsilon \beta [(V^0+V^1)\Delta|V^0+V^1|^2 -  V^0\Delta|V^0|^{2p}]
      \\
  ~~ -  \epsilon  \rho  (V^0 \Delta W^1 +V^1 \Delta W^0 + V^1 \Delta W^1)
    =  i \sqrt{b} \left(\frac{1}{p}-1\right)V^0-i(V^0+V^1)_\tau,   \\
   \Delta_{\nu} W^1=(|V^0+V^1|^2-|V^0|^{2p})_{\xi_1 \xi_1}+ \epsilon \Delta (|V^0+V^1|^2- |V^0|^{2p})_{\xi_1 \xi_1}.
 \end{cases}
 \end{align*}
By the mean value theorem, $|V^0|^2-|V^0|^{2p} \approx (1-p) |V^0|^2 \ln|V^0|^2$  due to the fact $p\rightarrow 1^+$ as $\tau\rightarrow \infty$.
Also we assume that, as $\tau\rightarrow \infty$, $|V^1|\ll |V^0|$ and $|W^1| \ll |W^0|$.
Then using (\ref{app}), to the lowest order, as $\tau\rightarrow \infty$, the above system reduces to
\begin{align*}
\begin{cases}
  \Delta V^1-V^1+\beta S^2(2V^1+\bar V^1) + \beta (1-p) (S^3 \ln S^2)-\rho (W^1 S + X V^1)   \\
   \hspace{1 in}= i \sqrt{b} \left(\frac{1}{p}-1 \right) S-i(b_{\tau} G+\epsilon_{\tau} H), \\
  \Delta_{\nu} W^1=[S(V^1+\bar V^1)+(1-p)(S^2\ln S^2)]_{\xi_1\xi_1}.
\end{cases}
\end{align*}
Substituting $V^1=V_1+iV_2$ yields
\begin{align}  \label{v122}
\begin{cases}
  \Delta V_1-V_1+3 \beta S^2 V_1 -\rho (W^1 S + X V_1) =   \beta (p-1) (S^3 \ln S^2), \\
\Delta V_2- V_2 + \beta S^2 V_2 - \rho X V_2 = \sqrt{b} \left(\frac{1}{p}-1 \right) S - (b_\tau G+\epsilon_\tau H), \\
  \Delta_{\nu} W^1- 2 (S V_1)_{\xi_1 \xi_2}= (1-p) (S^2\ln S^2)_{\xi_1\xi_1}.
\end{cases}
\end{align}
Note that (\ref{v122})$_2$ (the 2nd equation in (\ref{v122})) is decoupled from (\ref{v122})$_1$ and (\ref{v122})$_3$. Concerning the system comprised of equations (\ref{v122})$_1$ and (\ref{v122})$_3$, the existence of solutions again requires the right-hand side of the system be orthogonal to the kernel of the adjoint of the operator arising in the left-hand side, which is spanned by the vectors given in (\ref{GYdualsol}). Therefore, the solvability condition of the system comprised of equations (\ref{v122})$_1$ and (\ref{v122})$_3$ reads
\begin{align*}
\frac{1}{4} \beta (p-1) \int_{\mathbb R^2} (S^4)_{\xi_j} \ln S^2 \; d\xi_1 d\xi_2  +
\frac{\rho}{2}(p-1)\int_{\mathbb R^2} X (S^2 \ln S^2)_{\xi_j} \; d\xi_1 d\xi_2=0,
\end{align*}
 for $j=1,\,2$, which is satisfied provided $S$ is symmetric and $X$ is even in $\xi_1$ and $\xi_2$. Also notice that, due to (\ref{SX}), $S$ satisfies the left-hand side of equation (\ref{v122})$_2$, and it follows that the solvability condition for (\ref{v122})$_2$ reads
\begin{align}    \label{sovl}
\int_{\mathbb{R}^2} \left[\sqrt{b} \left(\frac{1}{p}-1 \right) S^2-b_{\tau} SG - \epsilon_{\tau} SH  \right]\,d\xi_1d\xi_2=0.
\end{align}
From Appendix B, we know that
\begin{align}    \label{B'}
&C_1=\int_{\mathbb{R}^2}SG\,d\xi_1d\xi_2 =  \frac{1}{16} \int_{\mathbb R^2} |\xi|^2 S^2 \, d\xi_1 d\xi_2 >0  \notag\\
&C_2=\int_{\mathbb{R}^2}SH\,d\xi_1d\xi_2=\frac{1}{4}\left( \beta -  \frac{2\rho}{1+\nu}     \right) \int_{\mathbb{R}^2}|\nabla S^2|^2\,d\xi_1d\xi_2 >0
\end{align}
since $\rho<0$ and $\beta>0$. Thus (\ref{sovl}) can be written as
 \begin{equation*}
 C_1 b_\tau+C_2\epsilon_\tau + \sqrt{b} \left(1-\frac{1}{p}\right)\|S\|_2^2  =0 .
 \end{equation*}
Since $p>1$ and $b>0$, we obtain that
\begin{align*}
C_1 b_\tau + C_2\epsilon_\tau <0.
\end{align*}
Integrating from $0$ to $\tau$ gives $$C_1 b + C_2 \epsilon < C_3,$$
for some constant $C_3= (C_1 b + C_2 \epsilon)|_{\tau=0}$. Recall that $b=a^2 + a_{\tau}$, with $a=-L_t L$ and $\tau=\int_0^t \frac{1}{L^2(s)} \; ds$, thus $b=-L^3 L_{tt}$. Also recall that $\epsilon=\frac{\alpha^2}{L^2}$. It follows that
\begin{align}
-C_1 L^3 L_{tt}+C_2 \frac{\alpha^2}{L^2} < C_3\;.
\end{align}
The above can be written as $L_{tt} > \frac{C_2 \alpha^2}{C_1} L^{-5} - \frac{C_3}{C_1} L^{-3}$, then multiply both sides by $2L_t<0$, we obtain
$(L_t^2)_t <  - \frac{C_2 \alpha^2}{2C_1} (L^{-4})_t  +  \frac{C_3}{C_1} (L^{-2})_t$. Integrating from 0 to $t$ gives
\begin{align}
L^2 L_t^2 <  - \frac{C_2 \alpha^2}{2C_1}\frac{1}{L^2} +  \frac{C_3}{C_1} + C_4 L^2
\end{align}
where $C_4=L_t^2(0) + \frac{C_2 \alpha^2}{2C_1}L^{-4}(0)   -  \frac{C_3}{C_1} L^{-2}(0)$\;. Therefore the scaling factor $L$ can not approach zero since $C_1$, $C_2$, $\alpha>0$. This explains the prevention of the singularity formation, at this leading order in the expansion.

\begin{remark} A similar procedure for handling singularities can also be applied to the RDS1 system (\ref{dsh1}). When $\epsilon \rho V \Delta W$
and $\epsilon \Delta (|V|^2)_{\xi_1  \xi_1}$ are neglected in (\ref{rescaled}), then $C_2$ defined in (\ref{B'}) becomes $C_2=\frac{\beta}{4}\int_{\mathbb{R}^2}|\nabla S^2|^2\,d\xi_1d\xi_2>0$, since $\beta >0$. Furthermore, for the RDS2 system (\ref{dsh2}), when $\epsilon \beta V \Delta |V|^2$ is neglected in (\ref{rescaled}), we have $C_2=-\frac{\rho}{2(\nu+1)}\int_{\mathbb{R}^2}|\nabla S^2|^2\,d\xi_1d\xi_2>0$, since $\rho<0$. Therefore, these regularizations also prevent the singularity formation of the DSE (\ref{ds}).
\end{remark}

\begin{appendix}
\section*{Appendix A}
 \setcounter{equation}{0}
 \renewcommand{\theequation}{A.\arabic{equation}}
 \setcounter{theorem}{0}
 \renewcommand{\thetheorem}{A.\arabic{theorem}}

The aim of this Appendix is to state some well-known results in the theory of the Schr\"odinger equation concerning the operators
$G_0\psi(t)=e^{it \Delta}\psi$ and
$Gf(t)=\int_0^t e^{ i (t-s) \Delta }  f(s) \,ds$ in the 2-dimensional space (see, e.g., \cite{ghidaglia,kat,sul}):

\begin{lemma} \label{lgog1}
Let $r\in[2,\infty)$, $q\in(2,\infty)$, such that $\frac{1}{q}+\frac{1}{r}=\frac{1}{2}$. Then the following estimates hold:
\begin{eqnarray*}
&& \hspace{-1.2cm} \|G_0\psi\|_{L^q(\mathbb{R};L^r)}\leq \gamma \|\psi\|_2,\;\;
\|G_0\psi\|_{L^\infty(\mathbb{R};L^2)}\leq \gamma\|\psi\|_2, \\
&& \hspace{-1.2cm}  \|Gf\|_{L^q(\mathbb{R};L^r)}\leq \gamma\|f\|_{L^{1}(\mathbb{R};L^2)},  \;\;
\|Gf\|_{L^q(\mathbb{R};L^r)}\leq \gamma\|f\|_{L^{q'}(\mathbb{R};L^{r'})}, \\
&& \hspace{-1.2cm} \|Gf\|_{L^\infty(\mathbb{R};L^2)}\leq \gamma\|f\|_{L^{q'}(\mathbb{R};L^{r'})}.
\end{eqnarray*}
Here $q'$ and $r'$ are the dual pair of $q$ and $r$, respectively.
\end{lemma}

\bigskip

Recall the spaces $X'$ and $Y'$ are defined in (\ref{Yprime}), and the spaces $\bar X$ and $\bar Y$ are defined in (\ref{Ybar}).
\begin{lemma}\label{lgbound}
 $G_0$ is bounded from $L^2$ into $\bar{X}$ and  bounded from $H^1$ into $\bar{Y}$. $G$ is bounded from $X'$ into $\bar{X}$ and bounded from $Y'$ into $\bar{Y}$.
 The associated norms are independent of  $T$.
 \end{lemma}

\bigskip

Recall the spaces $Z$, $\bar Z$ and $Z'$ are defined in (\ref{Z}), (\ref{barZ}) and (\ref{Zprime}), respectively.
\begin{lemma}  \label{proZ}
$G_0$ is bounded from $H^2$ into $\bar Z$ and $G$ is bounded from $Z'$ into $\bar Z$ such that
\begin{align*}
& \|G_0 \psi\|_Z  \leq \gamma \|\psi\|_{H^2}    \\
& \|Gf\|_{Z} \leq (2\gamma+1) \|f\|_{Z'},  \text{\;\;if\;\;}  T\leq 1.
\end{align*}
\end{lemma}

 \section*{Appendix B}
 \setcounter{equation}{0}
 \renewcommand{\theequation}{B.\arabic{equation}}
This appendix is aimed to prove
\begin{align}
&\int_{\mathbb{R}^2}SG\,d\xi_1d\xi_2 =  \frac{1}{16} \int_{\mathbb R^2} |\xi|^2 S^2 \, d\xi_1 d\xi_2    \label{B0} \\
&\int_{\mathbb{R}^2}SH\,d\xi_1d\xi_2 = \frac{1}{4}\left(\beta - \frac{2\rho }{1+\nu}\right)\int_{\mathbb R^2} |\nabla S^2|^2 \; d\xi_1 d\xi_2   \label{B}
\end{align}
which were introduced in section \ref{modulation}.

The proofs for these two formulas are similar. So we only justify (\ref{B}) in details. Our argument follows the approach in \cite{pap}.

Recall that $(S,X)$ satisfies
   \begin{align}    \label{B1}
  \begin{cases}
  \Delta S -S +\beta S^3 -\rho S X = 0\\
  \Delta_{\nu} X-(S^2)_{\xi_1 \xi_1} = 0\;,
  \end{cases}
  \end{align}
and $(H,Z)$ satisfies
 \begin{align}   \label{B2}
 \begin{cases}
 \Delta H-H+3\beta H S^2-\rho (SZ + HX) = -\beta S \Delta (S^2) + \rho S \Delta X  \\
 \Delta_{\nu} Z- 2(SH)_{\xi_1\xi_1} = \Delta (S^2)_{\xi_1\xi_1}.
  \end{cases}
 \end{align}

Multiplying (\ref{B1})$_1$ by $H$, (\ref{B2})$_1$ by $S$, subtracting and integrating over $\mathbb{R}^2$, we obtain
\begin{align}   \label{B22}
\int_{\mathbb R^2}  \left(2\beta S^3 H - \rho S^2 Z + \beta  S^2 \Delta(S^2) - \rho S^2  \Delta X  \right) d\xi_1 d\xi_2  = 0\;.
\end{align}
Also, multiplying (\ref{B1})$_1$ by $(\xi_1,\xi_2) \cdot \nabla H$, (\ref{B2})$_1$ by $(\xi_1,\xi_2) \cdot \nabla S$,
adding and integrating over $\mathbb{R}^2$, it follows that
\begin{align}  \label{B3}
&\int_{\mathbb R^2} \left( 4SH - 2\beta S^3 H   +  \rho SH (2X+\xi_1 X_{\xi_1} +  \xi_2 X_{\xi_2})
+   \frac{\rho S^2}{2} (2Z+ \xi_1 Z_{\xi_1} +  \xi_2 Z_{\xi_2})  \right)  \; d\xi_1 d\xi_2       \notag\\
& = \frac{1}{2} \int_{\mathbb R^2} \left([\xi_1(S^2)_{\xi_1} + \xi_2(S^2)_{\xi_2}](-\beta \Delta(S^2)+\rho \Delta X)\right) \; d\xi_1 d\xi_2\;.
\end{align}

At this stage, let us define
\begin{equation*}
(X_1)_{\xi_1\xi_1}=X,\hspace{1cm} (Z_1)_{\xi_1\xi_1}=Z. \label{xz}
\end{equation*}
Multiplying (\ref{B1})$_2$ by $(\xi_1,\xi_2) \cdot \nabla Z_1$ and integrating over $\mathbb R^2$ yield
\begin{align}    \label{B5}
&\int_{\mathbb R^2}  (X-S^2) ( 2Z + \xi_1 Z_{\xi_1} + \xi_2 Z_{\xi_2} )   \; d\xi_1 d\xi_2 \notag\\
&+ \nu \int_{\mathbb R^2}  X  [2(Z_1)_{\xi_2\xi_2} + \xi_1 (Z_1)_{\xi_1 \xi_2 \xi_2}  + \xi_2  (Z_1)_{\xi_2 \xi_2 \xi_2} ] \; d\xi_1 d\xi_2 =0 \;.
\end{align}
Notice that
\begin{align*}
\int_{\mathbb R^2} \nu X (Z_1)_{\xi_2 \xi_2} \,d\xi_1 d\xi_2=
\int_{\mathbb R^2} \nu X_{\xi_2 \xi_2} Z_1 \,d\xi_1 d\xi_2
&=\int_{\mathbb R^2} [(S^2)_{\xi_1 \xi_1} - X_{\xi_1 \xi_1} ] Z_1 \,d\xi_1 d\xi_2  \notag\\
&=\int_{\mathbb R^2} (S^2 - X ) Z   \,d\xi_1 d\xi_2\, ,
\end{align*}
which can be substitute into (\ref{B5}), and it follows that
\begin{align}   \label{B55}
\int_{\mathbb R^2}  (X-S^2) ( \xi_1 Z_{\xi_1} + \xi_2 Z_{\xi_2} )   \; d\xi_1 d\xi_2 + \nu \int_{\mathbb R^2}  X  [\xi_1 (Z_1)_{\xi_1 \xi_2 \xi_2}  + \xi_2  (Z_1)_{\xi_2 \xi_2 \xi_2} ] \; d\xi_1 d\xi_2 =0 \;.
\end{align}
Also, multiplying (\ref{B2})$_2$ by $(\xi_1,\xi_2) \cdot \nabla X_1$ and integrating yields
\begin{align}    \label{B6}
&\int_{\mathbb R^2} [Z - 2SH - \Delta(S^2)] (2X + \xi_1 X_{\xi_1} + \xi_2 X_{\xi_2}   )   \; d\xi_1 d\xi_2   \notag\\
   &+  \nu \int_{\mathbb R^2}  Z_{\xi_2 \xi_2}[\xi_1(X_1)_{\xi_1} + \xi_2 (X_1)_{\xi_2} ]       \; d\xi_1 d\xi_2 =0\,.
\end{align}
Now substituting
\begin{align*}
\int_{\mathbb R^2} Z_{\xi_2 \xi_2}[\xi_1(X_1)_{\xi_1} + \xi_2 (X_1)_{\xi_2} ]  \; d\xi_1 d\xi_2
=- \int_{\mathbb R^2} X [\xi_1 (Z_1)_{\xi_1 \xi_2 \xi_2} + \xi_2 (Z_1)_{\xi_2 \xi_2 \xi_2}]   \; d\xi_1 d\xi_2\,
\end{align*}
into (\ref{B6}) yields
\begin{align}   \label{B66}
&\int_{\mathbb R^2} [Z - 2SH - \Delta(S^2)] (2X + \xi_1 X_{\xi_1} + \xi_2 X_{\xi_2}   )   \; d\xi_1 d\xi_2   \notag\\
&- \nu \int_{\mathbb R^2} X [\xi_1 (Z_1)_{\xi_1 \xi_2 \xi_2} + \xi_2 (Z_1)_{\xi_2 \xi_2 \xi_2}]   \; d\xi_1 d\xi_2 =0\, .
\end{align}
Adding (\ref{B55}) and (\ref{B66}) gives us
\begin{align*}
&\int_{\mathbb R^2}  (X-S^2) ( \xi_1 Z_{\xi_1} + \xi_2 Z_{\xi_2} )   \; d\xi_1 d\xi_2    \notag\\
&+ \int_{\mathbb R^2} [Z - 2SH - \Delta(S^2)] (2X + \xi_1 X_{\xi_1} + \xi_2 X_{\xi_2}   )   \; d\xi_1 d\xi_2  =0\, ,
\end{align*}
and since
\begin{align*}
\int_{\mathbb R^2}  X( \xi_1 Z_{\xi_1} + \xi_2 Z_{\xi_2} )
= -  \int_{\mathbb R^2} Z (2X + \xi_1 X_{\xi_1} + \xi_2  X_{\xi_2} )   \; d\xi_1 d\xi_2\, ,
\end{align*}
we obtain that
\begin{align}  \label{B7}
\int_{\mathbb R^2}S^2( \xi_1 Z_{\xi_1} + \xi_2 Z_{\xi_2} )   \; d\xi_1 d\xi_2 + \int_{\mathbb R^2} [2SH +\Delta(S^2)] (2X + \xi_1 X_{\xi_1} + \xi_2 X_{\xi_2}   )   \; d\xi_1 d\xi_2  =0\, .
\end{align}
Multiplying (\ref{B7}) by $\frac{\rho}{2}$ and substituting the result into the sum of (\ref{B22}) and (\ref{B3}), it follows that
\begin{align} \label{B8}
&\int_{\mathbb R^2} \left(4SH + \beta S^2 \Delta(S^2) - 2\rho S^2 \Delta X  -\frac{\rho}{2} (\Delta (S^2)) (\xi_1 X_{\xi_1} + \xi_2 X_{\xi_2} ) \right)  \; d\xi_1 d\xi_2       \notag\\
& = \frac{1}{2} \int_{\mathbb R^2} \left([\xi_1(S^2)_{\xi_1} + \xi_2(S^2)_{\xi_2}](-\beta \Delta(S^2)+\rho \Delta X)\right) \; d\xi_1 d\xi_2\;.
\end{align}
Note that
\begin{align*}
&\int_{\mathbb R^2} [\xi_1(S^2)_{\xi_1} + \xi_2 (S^2)_{\xi_2} ] \Delta(S^2)\; d\xi_1 d\xi_2 \notag\\
&=\int_{\mathbb R^2} [\xi_1(S^2)_{\xi_1}(S^2)_{\xi_1\xi_1} + \xi_2(S^2)_{\xi_2}(S^2)_{\xi_1\xi_1}
+\xi_1(S^2)_{\xi_1}(S^2)_{\xi_2\xi_2}+\xi_2(S^2)_{\xi_2}(S^2)_{\xi_2\xi_2}] \; d\xi_1 d\xi_2 \notag\\
&=\int_{\mathbb R^2} \left[-\frac{1}{2}((S^2)_{\xi_1})^2 + \frac{1}{2} ((S^2)_{\xi_1})^2
+ \frac{1}{2} ((S^2)_{\xi_2})^2 -  \frac{1}{2} ((S^2)_{\xi_2})^2 \right] \; d\xi_1 d\xi_2=0\,.
\end{align*}
Consequently, (\ref{B8}) can be reduced to
\begin{align}  \label{B9}
&4\int_{\mathbb R^2} SH \; d\xi_1 d\xi_2   \notag\\
&=\beta \int_{\mathbb R^2} |\nabla S^2|^2 \; d\xi_1 d\xi_2 +  2\rho \int_{\mathbb R^2} S^2 \Delta X  \; d\xi_1 d\xi_2  \notag\\
& \hspace{0.1 in} +\frac{\rho}{2}\int_{\mathbb R^2} \left(\Delta (S^2) (\xi_1 X_{\xi_1} + \xi_2 X_{\xi_2} )
+[\xi_1(S^2)_{\xi_1} + \xi_2(S^2)_{\xi_2}]  \Delta X \right) \; d\xi_1 d\xi_2   \notag\\
&= \beta \int_{\mathbb R^2} |\nabla S^2|^2 \; d\xi_1 d\xi_2  + 2\rho\int_{\mathbb R^2}  S^2 \Delta X \; d\xi_1 d\xi_2\;.
\end{align}
Since $S$ and $X$ are symmetric and $\Delta_{\nu} X=(S^2)_{\xi_1 \xi_1}$, we obtain that $(1+\nu)\Delta X =  \Delta S^2$ which implies that
\begin{align}     \label{B10}
\int_{\mathbb R^2}  S^2 \Delta X \; d\xi_1 d\xi_2 =  -\frac{1}{1+\nu} \int_{\mathbb R^2}  |\nabla S^2|^2  \; d\xi_1 d\xi_2\,.
\end{align}
Substituting (\ref{B10}) into (\ref{B9}) yields
\begin{align*}
\int_{\mathbb R^2} SH \; d\xi_1 d\xi_2   =  \frac{1}{4}\left(\beta - \frac{2\rho }{1+\nu}\right)\int_{\mathbb R^2} |\nabla S^2|^2 \; d\xi_1 d\xi_2\,.
\end{align*}

\end{appendix}

 \section*{Acknowledgments} I. H. would like to thank the University of California--Irvine for the kind hospitality, where part of this work was completed.
 The work of I. H. was partially supported by TÜB{\.I}TAK (Turkish Scientific and Technological Research Council). The work of E. S. T. was supported in part by the ONR grant N00014--15--1--2333 and the NSF grants DMS--1109640 and DMS--1109645.

\end{document}